\documentclass[12pt, a4paper, leqno]{amsart}
\usepackage[utf8]{inputenc}
\usepackage{amsmath}
\usepackage{amsfonts}
\usepackage{amssymb}
\usepackage{times}
\usepackage[T1]{fontenc} 
\usepackage{url}
\usepackage[dvipsnames]{xcolor}
\usepackage[colorlinks, allcolors=RedViolet]{hyperref} 
\usepackage{graphicx}
\usepackage{esint}

\let\oldmarginpar\marginpar
\renewcommand\marginpar[1]{\-\oldmarginpar[\raggedleft\footnotesize #1]%
{\raggedright\footnotesize #1}}

\usepackage{amsthm}
\theoremstyle{plain}
\newtheorem{thm}[equation]{Theorem}
\newtheorem{lem}[equation]{Lemma}
\newtheorem{lemma}[equation]{Lemma}
\newtheorem{prop}[equation]{Proposition}
\newtheorem{cor}[equation]{Corollary}

\theoremstyle{definition}
\newtheorem{defn}[equation]{Definiton}
\newtheorem{assumptions}[equation]{Assumption}

\theoremstyle{remark}

\newtheorem{remark}[equation]{Remark}

\numberwithin{equation}{section}

\newcommand{\R}{\mathbb{R}}
\newcommand{\N}{\mathbb{N}}

\newcommand{\Rn}{\mathbb{R}^n}
\def\osc{\operatornamewithlimits{osc}}
\def\essinf{\operatornamewithlimits{ess\,inf}}

\renewcommand{\phi}{\varphi}
\renewcommand{\rho}{\varrho}
\renewcommand{\epsilon}{\varepsilon}
\renewcommand{\vartheta}{\theta}

\def\le{\leqslant}
\def\leq{\leqslant}
\def\ge{\geqslant}
\def\geq{\geqslant}

\def\esssup{\operatornamewithlimits{ess\,sup}}
\def\osc{\operatornamewithlimits{osc}}

\def\osc{\qopname\relax o{osc}}

\def\spt{\qopname\relax o{spt}}

\def\px{{p(\cdot)}}

\def\phix{{\phi}}

\def\loc{{\rm loc}}

\newcommand{\ainc}[1]{\hyperref[defn:aInc]{{\normalfont(aInc){\ensuremath{_{#1}}}}}}
\newcommand{\adec}[1]{\hyperref[defn:aDec]{{\normalfont(aDec){\ensuremath{_{#1}}}}}}
\newcommand{\adeci}[1]{\hyperref[defn:aDeci]{{\normalfont(aDec){\ensuremath{_{#1}^\infty}}}}}
\newcommand{\azero}{\hyperref[defn:a0]{{\normalfont(A0)}}}
\newcommand{\aone}{\hyperref[defn:a1]{{\normalfont(A1)}}}
\newcommand{\aonen}{\hyperref[defn:a1n]{{\normalfont(A1-\ensuremath{n})}}}

\newcommand{\Phiw}{\Phi_\text{\rm w}}

\date{\today}

\usepackage{color}
\definecolor{blau}{rgb}{0.1,0.0,0.9}

\newcounter{komcounter}
\numberwithin{komcounter}{section}

\begin{document}

\title[Hölder continuity of \texorpdfstring{$\omega$}{omega}-
minimizers]
{Hölder continuity of \texorpdfstring{$\omega$}{omega}-minimizers 
of functionals with generalized Orlicz growth}


\author{Petteri Harjulehto}
 \address{Petteri Harjulehto,
 Department of Mathematics and Statistics,
FI-20014 University of Turku, Finland}
\email{\texttt{petteri.harjulehto@utu.fi}}

\author{Peter Hästö}
 \address{Peter Hästö, Department of Mathematics and Statistics,
FI-20014 University of Turku, Finland, and Department of Mathematics,
FI-90014 University of Oulu, Finland}
\email{\texttt{peter.hasto@oulu.fi}}

\author{Mikyoung Lee}
 \address{Mikyoung Lee, Department of Mathematics, Pusan National University, Busan 46241, Republic of Korea}
\email{\texttt{mikyounglee@pusan.ac.kr}}





\begin{abstract}
We show local Hölder continuity of quasiminimizers of 
functionals with non-standard (Musielak--Orlicz) 
growth. Compared with previous results, we 
cover more general minimizing functionals and need fewer assumptions. 
We prove Harnack's inequality and a Morrey type estimate for quasiminimizers. 
Combining this with Ekeland's variational principle, we obtain local Hölder 
continuity for $\omega$-minimizers.
\end{abstract}

\maketitle

\section{Introduction}

Generalized Orlicz spaces have recently attracted increasing intensity (cf.\ 
Section~\ref{sect:spaces}). The results have also been applied to the study of 
differential equations with non-standard growth 
(e.g.\ \cite{ChlGZ18,GwiSZ18, HarHK17, HarHK16, HasO_pp18}). 
In \cite{HarHT17}, the first two authors and Toivanen gave the first proof 
of Harnack's inequality for solutions under generalized Orlicz growth. 
We start this paper giving a more sophisticated proof of this inequality, 
with better dependence of the constants on the structure of the equation. 
In contrast the the earlier result, this improved Harnack inequality can be 
applied to prove the H\"older continuity of $\omega$-minimzers, which is the 
second part of this paper.

In the fields of partial differential equations and the calculus of variations, 
there has been much research on non-standard growth problems 
(e.g.\ \cite{AM2,AM0,BO1,Mar89,Mar91}), 
such as the non-autonomous minimization problem
$$
\min_{v\in W^{1,1}} \int_\Omega F(x,\nabla v)\, dx
$$
where $F$ satisfies $(p,q)$-growth conditions, that is, 
$|\xi|^p - 1 \lesssim F(x,\xi)\lesssim |\xi|^q + 1, \ q>p$. 
Zhikov \cite{Zhi86, Zhi95} considered special cases as models of anisotropic 
materials and the so-called Lavrentiev phenomenon. In \cite{Zhi95}, he 
proposed model problems including
$$
F(x,\xi)\approx |\xi|^{p(x)},\ \ 1< \inf p\le \sup p <\infty,
$$
and
\begin{equation}\label{eq:doublephase}
F(x,\xi)\approx |\xi|^{p}+a(x)|\xi|^q,\ \ 1<p\leq q<\infty,\ \ a\geq 0.
\end{equation}
For the first, so-called \textit{variable exponent} case, the exponent of $|\xi|$ is a 
function of the $x$-variable which is usually assumed to be $\log$-H\"older continuous, and it 
describes 
various phenomena, for example electrorheological fluids \cite{Ruz00} and image 
restoration \cite{CheLR06,HarHLT13}, with growth continuously changing with 
respect to the position. 
The second, so-called \textit{double phase} case describes for instance composite 
materials or mixtures. Here, a discontinuous phase transition occurs on the 
border between constituent materials. In a series of papers, Baroni, Colombo 
and Mingione \cite{BarCM15, BarCM18, 
ColM15a, ColM15b, ColM16} have studied regularity properties of minimizers of 
these problems, see also \cite{ByuO17, ByuRS18, EleMM16a, EleMM16b, EleMM18, Ok17, ZhaR18}. 
Cupini, Pasarelli di Napoli and co-authors \cite{CloGHP_pp18, CupGGP18} 
have considered the variant of the double phase functional
\begin{equation}\label{eq:degen}
F(x,\xi)=(|\xi|-1)_+^p + a(x) (|\xi|-1)_+^q
\end{equation}
with $(s)_+:= \max\{s,0\}$, which is degenerate for small positive values of the 
gradient (see also \cite[Section~7.2]{HarH19} on how this functional fits into the generalized Orlicz framework). 
Furthermore, minimizers of borderline functionals like 
\[
F(x,\xi)=|\xi|^{p(x)}\log(e+|\xi|) \ \ \text{and}\ \ F(x,\xi)=|\xi|^p+a(x)|\xi|^p\log(e+|\xi|)
\]
have been recently studied, see for instance 
\cite{BarCM16,ByuO17,GiaP13,Ok16,Ok16b,Ok16c}.
We stress that \textit{all of these special cases are covered by the results
in this paper} (cf.\ \cite[Section~7.2]{HarH19}). In many cases the results of this paper 
are new even in the special cases.

The notion of an $\omega$-minimizer, sometimes called \textit{almost 
minimizer}, was introduced by Anzellotti \cite{Anz83}, and an analogous 
notion was originally given by Almgren \cite{Alm76} in the context of 
geometric measure theory. 
It was motivated by the fact that minimizers of constrained problems can turn 
out to be $\omega$-minimizers of unconstrained problems. 
For instance, minimizers of energy functionals with volume 
constraints or obstacles are $\omega$-minimizers, where the function 
$\omega$ is determined by the properties of the constraint \cite{Anz83, 
DuzGG00}. In this regard, the notion of an $\omega$-minimizer 
is useful and has been widely studied in the calculus of variations.
   
Regularity theory for minimizers has been extended to $\omega$-minimizers 
under suitable decay conditions on the function 
$\omega$ in for instance \cite{Anz83, DuzGG00,GobZ08,KriM05}, see also \cite{Min06} 
for a survey.
In particular, H\"older continuity of $\omega$-minimizers  was established by 
Dolcini--Esposito--Fusco \cite{DolEF96} in the standard $p$-growth case and 
later by Esposito--Mingione \cite{EspM99} in more general cases. Recently, it 
was also proved in double phase and Orlicz growth cases by Ok \cite{Ok17}.\footnote{This paper contains 
some problems in the proofs. With the assistance of Jihoon Ok, we have also managed improved the 
proofs to circumvent these problems.} 

We prove an extension of these results to the generalized Orlicz growth case. 
Our energy functional is given by 
\begin{equation}\label{mainfnal}
\mathcal{F}(u, \Omega) := \int_{\Omega} F(x, u, \nabla u) \,dx 
\end{equation}
where $ F: \Omega \times\R\times \Rn \to \R$ satisfies 
\begin{equation}\label{Hcon}
\nu\, \phi(x, |z|) \le F(x, t, z) \le N\, \big(\phi(x,|z|) + \Lambda \big) 
\end{equation}
for some $0<\nu \le N$ and $\Lambda \geq 0$. 
The exact definitions of the conditions in the following result 
are given in the next section; roughly,  \azero{}
restricts us to unweighted situations, \aone{} and \aonen{} are subtle continuity conditions 
and \ainc{} and \adeci{} exclude $L^1$- and $L^\infty$-type behavior, respectively.

\begin{thm}\label{mainthmV}
Let $\Omega\subset \Rn$ be a domain and $\phi\in\Phiw(\Omega)$ satisfy \azero{}, \ainc{} and \adeci{}. 
Let $u \in W^{1,\phi}_{loc}(\Omega)$ be an $\omega$-minimizer of $
\mathcal{F}$ and $(t, z)\to F(x, t, z)$ be continuous.
Assume that $\phi$ satisfies \aone{}, or that $u$ is bounded and $\phi$ 
satisfies \aonen{}. 
Then $u$ is locally H\"older continuous.
\end{thm}

The proof of this result is based on the variational technique described in 
\cite{DolEF96, Giu03}. The key idea is to find a quasiminimizer $w \in u+ 
W_0^{1,\phi}(Q_r)$ of the functional 
\[
\int_{Q_r} \phi(x,|\nabla w|)+  \Lambda \,dx,
\] 
which is comparable to our original $\omega$-minimizer $u$ 
of $\mathcal{F}$, by applying Ekeland’s variational principle with estimates 
depending on the constant $\Lambda$. 
From Harnack's inequality (Theorem~\ref{thm:Harnack}), 
it can be proved that the gradient of the quasiminimizer $w$ satisfies 
Morrey-type decay estimates (Section~\ref{sect:morrey}). 
A challenge compared to the classical case is that the constant in 
Harnack's inequality depends on $\Lambda_u$ and hence on $u$. 
However, we show that the natural bound $\Lambda_u\le |Q_r|^{-1}$ is 
sufficient to control the constant. 
Therefore, using the Morrey-type decay estimates, we can derive similar decay estimates of 
$\nabla u$, which implies H\"older continuity of $u$ (Section~\ref{sect:omega}). 
A further challenge worth mentioning is that 
moving between $\omega$-minimizers of $\phi$ and $\phi+\Lambda$ is not possible, 
so for this case we need to work directly with the condition \adeci{}. 

For the case \aonen{} (with $u$ bounded) 
we need to consider an alternative notion of minimizer 
called \textit{weak quasiminimizer} (cf.\ Definition~\ref{defminimizer}),
since we cannot otherwise guarantee boundedness of the quasiminimizer $w$ discovered by 
the Ekeland variational principle. This technique is adapted from \cite{Ok17}. 


\section{Generalized \texorpdfstring{$\Phi$}{Phi}-functions}\label{sect:phi-functions}

By $\Omega \subset \Rn$ we denote a bounded domain, i.e.\ an open and connected 
set. By $p':=\frac p {p-1}$ we denote the H\"older conjugate exponent 
of $p\in [1,\infty]$. The notation $f\lesssim g$ means that there exists a constant
$C>0$ such that $f\le C g$. The notation $f\approx g$ means that
$f\lesssim g\lesssim f$ whereas $f\simeq g$ means that 
$f(t/C)\le g(t)\le f(Ct)$ for some constant $C\ge 1$. 
By $c$ we denote a generic constant whose
value may change between appearances.
A function $f$ is \textit{almost increasing} if there
exists $L \ge 1$ such that $f(s) \le L f(t)$ for all $s \le t$
(more precisely, $L$-almost increasing).
\textit{Almost decreasing} is defined analogously.
By \textit{increasing} we mean that the inequality holds for $L=1$ 
(some call this non-decreasing), similarly for \textit{decreasing}. 

\begin{defn}\label{def:phi}
We say that $\phi: \Omega\times [0,\infty) \to [0,\infty]$ is a 
\textit{$\Phi$-prefunction} if the following hold:
\begin{itemize}
\item[(i)] For every $t \in [0,\infty)$ the function $x \mapsto \phi(x,t)$ is 
measurable.
\item[(ii)] For every $x \in \Omega$ the function $t \mapsto \phi(x,t)$ is 
increasing. 
\item[(iii)] $\lim\limits_{t \to 0^+} \phi(x,t) = \phi(x,0)=0$ and $
\lim\limits_{t \to\infty} \phi(x,t)= \infty$ for every $x \in \Omega$.
\end{itemize}
A $\Phi$-prefunction is a \textit{weak 
$\Phi$-function}, denoted by $\phi \in \Phiw(\Omega)$, if the following hold:
\begin{itemize}
\item[(iv)] The function $t \mapsto \frac{\phi(x,t)}{t}$ is $L$-almost 
increasing in $(0, \infty)$ for every $x\in \Omega$.
\item[(v)] The function $t \mapsto \phi(x,t)$ is 
left-continuous for every $x\in \Omega$. 
\end{itemize} 
\end{defn}

Since our weak $\Phi$-functions are not bijections, they are not 
strictly speaking invertible. However, 
by $\phi^{-1}(x, \cdot):[0,\infty)\to [0,\infty]$ we denote the 
\emph{left-inverse} of $\phi$: 
\[
\phi^{-1}(x, \tau) := \inf\{t \ge 0 : \phi(x, t)\ge \tau\}. 
\]
If $\phi$ is strictly increasing, then this is just the normal inverse function, 
but that is not a convenient assumption for us. 
Let $\phi \in \Phiw(\Omega)$. We say that $\phi$ satisfies 
\begin{itemize}
\item[(A0)]\label{defn:a0}
if there exists $\beta \in(0, 1]$ such that $\beta \le \phi^{-1}(x,1) \le 
\frac1{\beta}$ for a.e.\ $x \in \Omega$, 
or equivalently there exists $\beta \in(0,1]$ such that 
$\phi(x,\beta) \le 1\le \phi(x,\frac1\beta)$ for a.e.\ $x \in \Omega$
(see Corollary~3.7.4 in \cite{HarH19}).
\item[(A1)]\label{defn:a1}
if there exists $\beta\in (0,1)$ such that,
for every ball $B$ and a.e.\ $x,y\in B \cap \Omega$,
\[
\beta \phi^{-1}(x, t) \le \phi^{-1} (y, t) 
\quad\text{when}\quad 
t \in \bigg[1, \frac{1}{|B|}\bigg].
\]
\item[(A1-$n$)] \label{defn:a1n}
if there exists $\beta \in (0,1)$ such that, 
for every ball $B$ and a.e.\ $x,y\in B \cap \Omega$,
\[
\phi(x,\beta t) \le \phi(y,t)
\quad\text{when}\quad 
t \in \bigg[1, \frac{1}{|B|^{1/n}}\bigg].
\]
\item[(aInc)$_p$] \label{defn:aInc} if
$t \mapsto \frac{\phi(x,t)}{t^{p}}$ is $L$-almost 
increasing in $(0,\infty)$ for some $L\ge 1$ and a.e.\ $x\in\Omega$.
\item[(aDec)$_q$] \label{defn:aDec}
if
$t \mapsto \frac{\phi(x,t)}{t^{q}}$ is $L$-almost 
decreasing in $(0,\infty)$ for some $L\ge 1$ and a.e.\ $x\in\Omega$.
\item[(aDec)$_q^\infty$] \label{defn:aDeci}
if
$t \mapsto \frac{\phi(x,t)+1}{t^{q}}$ is $L$-almost 
decreasing in $(0,\infty)$ for some $L\ge 1$ and a.e.\ $x\in\Omega$.
\end{itemize} 

Moreover we say that $\phi$ satisfies \ainc{}, \adec{} or \adeci{} if it satisfies 
\ainc{p}, \adec{q} or \adeci{q}, respectively, for some $p>1$ or $q<\infty$. The 
condition \adeci{q} intuitively means that $t \mapsto \frac{\phi(x,t)}{t^{q}}$ is 
almost increasing for $t>T$ for some constant $T>0$. 

If $\phi$ satisfies \adec{}, then 
\begin{equation}\label{eq:almostId}
\phi^{-1}(x, \phi(x, t))\approx \phi(x, \phi^{-1}(x, t)) \approx  t.
\end{equation}
The growth of the inverse is closely tied to that of the function: 
$\phi$ satisfies \ainc{p} or \adec{q} 
if and only if 
$\phi^{-1}$ satisfies \adec{\frac 1p} or \ainc{\frac 1q}, respectively. 
For the proofs of these facts, see Section~2.3 in \cite{HarH19}.

By \cite[Proposition~4.1.5]{HarH19}, \aone{} implies that there
exists $\beta\in (0,1)$ such that
\[
\phi(x,\beta t) \le \phi(y,t)
\quad\text{when}\quad
\phi(y,t)\in [1 , \tfrac 1{|B|}]
\]
for almost every $x,y\in B \cap \Omega$ and 
every ball $B$ with $|B|\le 1$. Furthermore, if $\phi\in \Phiw$, then 
$\phi(\cdot ,1)\approx 1$ implies \azero{}, and if $\phi$ satisfies \adec{}, then \azero{} and $\phi(\cdot ,1)\approx 1$  are equivalent. 
In addition, when \adec{} holds we can multiply by 
constants in the range: $[a,\frac b{|B|}]$, $a,b>0$. 

The next lemma shows how we can use a trick to upgrade \adeci{} to \adec{} 
while preserving many other properties. 

\begin{lem}\label{lem:phi+t} 
Let $\phi \in \Phiw(\Omega)$ and define $\psi(x,t):= \phi(x,t) +t$. Then $\psi \in \Phiw(\Omega)$. Moreover,
\begin{enumerate}
\item[(a)] if $\phi$ satisfies \azero{}, then $\phi \le \psi \lesssim \phi +1$ and $\psi$ satisfies \azero{}; 
\item[(b)] if $\phi$ satisfies \adeci{q} and \azero{}, then $\psi$ satisfies \adec{q};
\item[(c)] if $\phi$ satisfies \aone{}, then $\psi$ satisfies \aone{};
\item[(d)] if $\phi$ satisfies \aonen{}, then $\psi$ satisfies \aonen{}. 
\end{enumerate}
\end{lem}
\begin{proof}
Checking the properties in Definition~\ref{def:phi}, 
we find that $\psi \in \Phiw(\Omega)$.

(a) The inequality $\phi\le \psi$ is immediate. 
Let $\phi$ satisfy \azero{} and assume first that $t >\frac{1}{\beta}$. Then we obtain by \azero{} and \ainc{1} that 
\[
\psi(x,t)=\phi(x, t) + t \le \phi(x, t) + \phi\big(x, \tfrac1\beta\big) t
\le  \phi(x, t) + \tfrac{L}{\beta} \phi(x,t) = \big(1+ \tfrac{L}{\beta} \big)\phi(x,t).
\]
If $t \le \frac1\beta$, then 
$\psi(x, t) \le \phi(x, t) +\frac1\beta \le (1+\frac1\beta) (\phi(x, t) +1)$. 
From the inequalities it follows that $\psi(x,\beta)\lesssim \phi(x,\beta)+1\le 2$ 
and $\psi(x,\frac1\beta)\ge 1$, and hence  \azero{} follows. 
 
(b) Let us then assume that $\phi$ satisfies \adeci{q} and \azero{}. 
If $t > s \ge \beta$, then by \adec{q}
\[
\begin{split}
\frac{\psi(x, t)}{t^q} 
&\lesssim
\frac{\phi(x,t)+1}{t^q}
\lesssim
\frac{\phi(x,s)+1}{s^q}
\lesssim \frac{\psi(x, s)}{s^q}.
\end{split}
\]
Let then $0<t\le \beta$. By \ainc{1}  and \azero{}, 
$\psi(x,t)\approx t$, so \adec{q} is clear in this range. 
The case $s\le \beta \le t$ follows by combining the previous cases.

(c) From the definition of left-inverse we directly see that 
$\psi^{-1}(x,t) \approx \min\{\phi^{-1}(x,t), t\}$. 
Thus we obtain by \aone{} of $\phi$ for $t \in \big[1, \frac1{|B|}\big]$ that
\[
\psi^{-1}(x, t) \approx \min\{\phi^{-1}(x,t), t\} 
\lesssim \min\{\phi^{-1}(y,t), t\}
\approx \psi^{-1}(y, t) .
\]

(d) Let $t \in \big[1, \frac{1}{|B|^{1/n}}\big]$.  By 
\aonen{} of $\phi$ we obtain
 \[
 \psi (x, \beta t) = \phi (x, \beta t) + \beta t
 \le \phi (y, t) + t = \psi(y, t). \qedhere
 \]
 \end{proof}

The Krylov--Safonov lemma used in the proof of Harnack's inequality works 
only for cubes, whereas \aone{} and \aonen{}-conditions have been defined with 
balls. However, a given cube $Q$ can be covered by a finite number, depending only 
on $n$, of balls $B_i$ with $|B_i|=|Q|$, and so the \aone{} or \aonen{} inequalities 
can be obtained in $Q$ by considering a chain of balls.
%

When $\phi_B^-(t) \le \frac 1{|B|}$ we will often use that \azero{}, \aone{} and 
\adec{} imply
\[
\phi_B^+(t)\lesssim \phi_B^-(t) + 1.
\]
Let us here give the details. If $\phi_B^-(t) \in [1, \frac 1{|B|}]$, then the 
inequality holds (without the $+1$) by \aone{} and \adec{}. If  $\phi_B^-(t) \le 1$, 
we instead use $\phi_B^+(t)\le \phi_B^+(1)\lesssim 1$ by \azero{} and \adec{}. 
Using same arguments we obtain the corresponding estimate for \aonen{}. 


\section{Generalized Orlicz spaces}\label{sect:spaces}

Generalized Orlicz spaces, also called Musielak--Orlicz spaces, have been actively 
studied over a long time. The basic example of a generalized Orlicz space was introduced by 
Orlicz \cite{Orl31} in 1931, and a major synthesis is due to Musielak \cite{Mus83} in 1983. 
Recent monographs on generalized Orlicz spaces are due to 
Yang, Liang and Ky \cite{YanLK17}, Lang and Mendez \cite{LanM19}, and 
the first two authors \cite{HarH19} 
focusing on Hardy-type spaces, functional analysis, and 
harmonic analysis, respectively; see also the survey article \cite{Chl18}. 
Generalized Orlicz spaces include as a special case classical Orlicz spaces that are well-known and have been extensively 
studied, see, e.g., the monograph \cite{RaoR91} and references therein. 

From this observation, we can roughly understand generalized Orlicz spaces as 
variable versions of Orlicz spaces with respect to the space variable $x$. 
The special case of variable exponent spaces $L^{\px}$ has been studied 
intensively over the last 20 years \cite{CruF13, DieHHR11, RadR15}. 
The reason that variable exponent research thrived while 
little harmonic analysis was done in generalized Orlicz spaces was the belief that 
many classical results can be obtained in the former setting but not the latter. 
A spate of recent articles (e.g.\ \cite{AhmCGY18, CapCF18, CruH18, 
HarH17, Has15, Kar18, MaeOS17, MaeMOS13, OhnS18, RafS17, YanYZ14}) 
has proved this belief to be unfounded.

Throughout the paper we write $\phi_B^+ (t):=\sup_{x\in B} \phi(x,t)$ and 
$\phi_B^- (t):=\inf_{x\in B} \phi(x,t)$ and abbreviate 
$\phi^\pm := \phi_\Omega^\pm$. Especially 
$\phi_B^-$ will be used countless times, since it enables us to apply the 
following Jensen-type inequalities. 
The function $\phi_B^-$ need not to be left-continuous, see 
\cite[Example 4.3.3]{HarH19} and hence it is not necessary a weak $\Phi$-function. 
However since it satisfies \ainc{} it is equivalent with a convex $\Phi$-function 
(independent of $x$) by \cite[Lemma~2.2.1]{HarH19}. 
This is used in the next lemma, where $\phi$ is independent of $x$, e.g.\ $\phi_B^-$.
We denote by $L^0(\Omega)$ the set of measurable functions in $\Omega$. 
By $f_D$ and $\fint_D f\, dx$ we denote the integral average of $f$ over $D$. 


\begin{lemma}\label{lem:aJensen}
Let $\phi$ be a $\Phi$-prefunction which satisfies \ainc{p} and \adec{q}, 
$D\subset \Rn$ be measurable with $|D| \in(0, \infty)$ and $f\in L^0(D)$. 
Then 
\[
\bigg( \fint_D |f|^p \,dx \bigg)^\frac1p 
\lesssim 
\phi^{-1}\bigg( \fint_D \phi(|f|) \,dx \bigg)
\lesssim
\bigg( \fint_D |f|^q \,dx \bigg)^\frac1q. 
\]
\end{lemma}
\begin{proof}
Let $\psi(t):= \phi(t^{1/p})$. Then $\psi$ satisfies \ainc{1} and 
so there exists a convex $\xi\in \Phiw$ with $\psi\simeq \xi$ with constant 
$\beta$ \cite[Lemma~2.2.1]{HarH19}. 
Since $\xi$ is 
convex, Jensen's inequality implies that 
\[
\phi \bigg( \beta^\frac2p \Big( \fint_D |f|^p \,dx \Big)^\frac1p  \bigg)
\le 
\xi\bigg( \beta \fint_D |f|^p \,dx \bigg)
\le 
\fint_D \xi(\beta |f|^p) \,dx
\le 
\fint_D \phi(|f|) \,dx. 
\]
Note that this inequality does not require \adec{}. 
The first inequality of the claim follows from this by \eqref{eq:almostId}. 

We know that $\phi^{-1}$ is increasing \cite[Lemma~2.3.9]{HarH19} and 
thus so is  $(\phi^{-1})^{q}$. Since $\phi$ satisfies \adec{q},
$\phi^{-1}$ satisfies \ainc{1/q} by \cite[Proposition~2.3.7]{HarH19} 
and $(\phi^{-1})^{q}$ satisfies \ainc{1}. Thus $(\phi^{-1})^{q}$ is a 
$\Phi$-prefunction. Hence by \cite[Lemma~2.2.1]{HarH19} there exists a convex
$\xi \in \Phiw$ such that $\xi \simeq (\phi^{-1})^{q}$. 
We obtain by Jensen's inequality
\begin{align*}
\phi^{-1}\bigg(\fint_D \phi(|f|) \, dx \bigg) 
&\approx \xi\bigg( \fint_D \phi(|f|)\, dx\bigg)^\frac1q
\le \bigg( \fint_D \xi(\phi(|f|))\, dx\bigg)^\frac1q 
\approx \bigg( \fint_D |f|^q\, dx\bigg)^\frac1q . \qedhere
\end{align*}
\end{proof}

Let $\phi \in \Phiw(\Omega)$. The \emph{generalized Orlicz space} (also known 
as the Musielak--Orlicz space) is defined as the 
set 
\[
L^{\phi}(\Omega) : = \Big\{ f \in L^0(\Omega) : \lim_{\lambda\to 0^+} 
\varrho_{\phi}(\lambda f) =0 \Big\}
\]
equipped with the (Luxemburg) norm
$$ \| f \|_{L^{\phi}(\Omega)} := \inf \Big\{ \lambda >0 : \varrho_{\phi} 
\Big( \frac{ f}{\lambda}\Big) \le 1\Big\},$$
where $ \varrho_{\phi}(f)$ is the \textit{modular} of $f \in L^0(\Omega)$ defined by 
\[
\varrho_{\phi}(f):=\int_{\Omega} \phi(x, |f(x)|)\,dx. 
\]

In many places, we make the following set of assumptions. However, this 
will be explicitly specified, as some results work also under fewer assumptions. 
Furthermore, all \textit{constants in our estimates depend only on the parameters 
in the assumptions and the dimension $n$}, unless something else is explicitly 
states. Specifically, these parameters are 
the constants $\beta$ and $L$, the exponents $p$ and $q$, 
the minimizing parameters $Q$ and $\omega$ (Definition~\ref{defminimizer})
and the structure constants $\nu$ and $N$ (from \eqref{Hcon}). 
However, the dependence on $\Lambda$ and the size of the cube $r$ will be made 
explicit, since we will need the cases $\Lambda\to \infty$ and $r\to 0$. 

\begin{assumptions}\label{ass:std}
The function $\phi\in \Phiw(Q_r)$ satisfies \ainc{p}, \adec{q}, \azero{} and 
one of the following holds for the function $u\in W^{1,\phi}(Q_r)$
\begin{enumerate}
\item 
$\phi$ satisfies \aone{} and $\rho_\phi(\nabla u)\le 1$; or
\item
$\phi$ satisfies \aonen{} and $u\in L^\infty(\Omega)$.
\end{enumerate}
\end{assumptions}

In the second case of the assumption, constants depend also on $\|u\|_\infty$. 
Note that the assumptions could be more symmetrical by assuming 
$\rho_\phi(\nabla u)< \infty$ in (1), in which case the constants would 
depend on $\rho_\phi(\nabla u)$, or, alternatively, $\|u\|_\infty\le 1$ in (2). 
However, it seems that the current versions are more natural to use. 

\bigskip

A function $u \in L^{\phi}(\Omega)$ belongs to the \emph{Orlicz--Sobolev space 
$W^{1,\phi}(\Omega)$} if its weak partial derivatives $\partial_1 u, \dots 
\partial_n u$ exist and belong to $L^{\phi}(\Omega)$. Furthermore, 
$W^{1,\phi}_0 (\Omega)$ is defined as the closure of $C^{\infty}_0(\Omega)$ in 
$W^{1,\phi}(\Omega)$. 

We will need the Sobolev--Poincar\'e inequality numerous times in this 
article, with either zero boundary values or with 
average zero. For the calculus of variations, inequalities in 
modular form, with an error term, are more useful than inequalities 
concerning norms (such as the ones in \cite{HarH17}). Furthermore, 
it is useful to have a constant exponent improvement $s>1$ in the 
integrability regardless of growth. 
Note that the exponent $s$ can be on the right-hand side or on the 
left-hand side, see Proposition~6.3.12 and Corollary 6.3.15  of 
\cite{HarH19}. In this paper we need the following versions.

\begin{thm}[Sobolev--Poincar\'e inequality] \label{thm:SP}
Let $B_r \subset \Rn$ be a ball or a cube with diameter $2r$. 
Let $\phi \in \Phiw(B_r)$ satisfy Assumption~\ref{ass:std}. 
For $1 \le s < \frac{n}{n-1}$, 
\begin{equation}\label{SPineq2}
\bigg(\fint_{B_r} \phi \bigg(x, \frac{|u|}r\bigg)^s\,dx 
\bigg)^{\frac1s}
\lesssim
\fint_{B_r} \phi (x, |\nabla u|) \,dx  + \frac{|\{\nabla u \neq 0\}\cap {B_r}|}
{|{B_r}|}
\end{equation}
for any $u \in W^{1,1}_0({B_r})$. 
If additionally $1\le s\le p$, then 
\begin{equation}\label{SPineq}
\fint_{B_r} \phi \bigg(x, \frac{|u-u_{B_r}|}{r}\bigg)\,dx 
\lesssim
\bigg(\fint_{B_r} \phi (x, |\nabla u|)^{\frac1s} \,dx \bigg)^{s} + 1
\end{equation}
for any $u \in W^{1,1}({B_r})$; in the case \aone{}, we need that 
$\| \nabla u \|_{\phi^{1/s}}\le M$, and the implicit constant depends on 
$M$.
The average $u_{B_r}$ can be replaced by $u_B$ for some ball or cube $B\subset B_r$ 
with $|B| > \mu |B_r|$, in which case the constant depends also on $\mu$.  
\end{thm}

The case \aone{} is covered by the Sobolev--Poincar\'e inequality in 
Proposition~6.3.12 and Corollary 6.3.15  of \cite{HarH19}, whereas 
the case of \aonen{} is new. 

\begin{proof}
We consider only bounded $u$ and \aonen{}. 
By  \cite[Lemma~2.2.1]{HarH19} there exists a convex
$\xi \in \Phiw$ such that $\xi \simeq \phi^-$. We apply \eqref{SPineq} 
to $\xi$, which satisfies \aone{} since it is independent of $x$:
\begin{align*}
\fint_{B_r} \phi^- \bigg(\frac{|u-u_{B_r}|}{r}\bigg)\,dx
&\lesssim
\fint_{B_r} \xi\bigg(\frac{|u-u_{B_r}|}{r}\bigg)\,dx \\
&\lesssim
\bigg(\fint_{B_r} \xi(|\nabla u|)^{\frac1s} \,dx \bigg)^{s} 
\lesssim
\bigg(\fint_{B_r} \phi^- (|\nabla u|)^{\frac1s} \,dx \bigg)^{s} \\
&\le
\bigg(\fint_{B_r} \phi (x, |\nabla u|)^{\frac1s} \,dx \bigg)^{s}.
\end{align*}
Furthermore, in this case the inequality $\| \nabla u \|_{\phi} <1$ is not 
needed, since \aone{} holds not only in $[1,\frac1{|B|}]$ but in $[0,\infty)$. 
On the left-hand side we use \aonen{}, \azero{} and \adec{} to estimate 
\[
\phi \bigg(x, \frac{|u-u_{B_r}|}{r}\bigg)
\lesssim
\phi^- \bigg(\beta\, \frac{|u-u_{B_r}|}{r}\bigg)+1,
\]
which concludes the proof in this case.  Note that in this case the constant depends on $\|u\|_\infty$. The other inequality can be proved 
similarly from \eqref{SPineq2}.
\end{proof}


\section{Quasiminimizers}\label{sect:quasiminimizers}

In a paper with Toivanen \cite{HarHT17}, the first two authors recently obtained the 
first results on regularity of quasiminimizers in the generalized Orlicz growth case. 
We showed a Harnack inequality and local H\"older continuity under assumptions 
\azero{}, \aone{}, \aonen{}, \ainc{} and \adec{}. 
The first aim of this paper is to improve and extend these results 
in several ways. 

The first main contribution 
of the current paper is to extend \cite{BarCM18} to the generalized Orlicz setting, 
and prove H\"older continuity assuming either \aone{} or \aonen{} and bounded $u$; 
In our earlier generalized Orlicz case result \cite{HarHT17}, 
we needed to assume both \aone{} and \aonen{}.
In addition, we here extend our previous results from \cite{HarHT17} in two ways.
Of greater importance is the inclusion of $+ \Lambda$ on the right-hand side: 
it allows us to move between \adeci{} and \adec{}  
and is crucial for applying quasiminimizer-results to prove regularity 
of $\omega$-minimizers. The \adeci{} assumption is a growth condition 
for large values of the gradient, a necessary change to handle \eqref{eq:degen} 
which does not satisfy a growth condition at the origin. 
A minor extension is that we allow $F$ to depend on $u$ and $\nabla u$, whereas 
the previous paper only allowed dependence on $|\nabla u|$. 

For quasiminimizers, our main result is the following Harnack inequality, 
which implies local H\"older continuity by well-known arguments. 
Note that the \aone{} and \aonen{} assumptions are essentially sharp, in view of the 
examples from the double phase case, cf.\ \cite{BalD_pp}. 

\begin{thm}[Harnack's inequality]\label{thm:Harnack}
Let $\Omega\subset \Rn $ be a domain and $\phi\in\Phiw(\Omega)$ satisfy  \azero{}, \ainc{} and \adeci{}. 
Let $u\in W^{1,\phix}_\loc(\Omega)$ be a non-negative local quasiminimizer of $\mathcal F$. 
Assume that $\phi$ satisfies \aone{}, or that $u$ is bounded and $\phi$ 
satisfies \aonen{}. 
If $Q_{2r} \subset \Omega$, then 
\[
\esssup_{x \in Q_r} u(x) \lesssim
\essinf_{x \in Q_r } u(x) +  (\phi^-_{Q_r})^{-1}(\Lambda+1)\,r
\]
provided $(\Lambda+1) |Q_{2r}| \le 1$ and $\int_{Q_{2r}} \phi(x, |\nabla u|) \, dx \le 1$. 
The implicit constant  depends only on the parameters from the assumptions,
the dimension $n$, and, in the case \aonen{}, on $\|u\|_\infty$; 
it is independent of $r$ and $\Lambda$. 
\end{thm}

The proof of this result (which continues through Sections~\ref{sect:esssup} and \ref{sect:essinf}) 
follows a different philosophy compared to 
our earlier paper \cite{HarHT17}: previously, much effort 
was directed at avoiding additional error terms which do not appear in the 
standard case, whereas now we focus on handling the error terms which 
appear. The reason is that the ``$+\Lambda$'' in \eqref{Hcon} 
as well as \adeci{} 
lead inevitably to similar additive error terms, so they must in any case 
be taken care of. These more streamlined proofs are made possible by new 
tools developed in the monograph \cite{HarH19}. It is especially 
worth mentioning the generalized Orlicz version of the Sobolev--Poincar\'e 
inequality (Theorem~\ref{thm:SP}) and the improved reverse H\"older 
inequality (Lemma~\ref{lem:jihoon}). While the proofs follow the well-known 
approach of De Giorgi, we found that they are very dependent on 
well set-up formulations (much more so that the variable exponent case): for instance the placement of $\tau$ on the left-hand side 
of \eqref{equ:DeGiorgi-class} and the estimate of 
$\frac{|A_j|}{|Q_j|}$ in the proof of Proposition~\ref{prop:esssupinf}. 
The main difficulty with the generalized Orlicz case is to move at suitable 
points in the proofs between $\phi(x,t)$ and $\phi_Q^-(t)$. 
This is accomplished via the Sobolev--Poincar\'e inequality or the 
Caccioppoli estimate. The former leads in the proof of Proposition~\ref{prop:esssupinf} 
to an additional term on the right-hand side, which can, however, be 
absorbed in the other terms in the specific cases needed for Harnack's 
inequality. Additional complications arise 
in several places because the Sobolev--Poincar\'e inequality holds only for functions 
with $\|\nabla u\|_{L^\phi(Q)}\le 1$. 

Recall that we define, for measurable $A\subset \Omega$, 
\[
\mathcal{F}(u, A) := \int_{A} F(x, u, \nabla u) \,dx. 
\]
By $Q_r$ we mean a cube with side length $r$ and faces parallel to the 
coordinate axes. Since we consider cubes, we speak of cubical minimizers, 
although spherical minimizers is a more common term for essentially the 
same thing. The results can also be adapted to spherical minimizers 
and $\omega$-minimizers defined in balls. 

\begin{defn}\label{defminimizer}
A function $u \in W^{1,\phi}_{loc}(\Omega)$ is called
\begin{itemize}
\item[(i)] 
a \textit{local quasiminimizer} of $\mathcal{F}$ if there exists $Q\ge 1$ such that 
\[
\mathcal{F}(u, \Omega'\cap \{u\neq v\}) \le Q \, \mathcal{F}(v,\Omega'\cap \{u\neq v\})
\]
for every open $\Omega' \Subset \Omega$ and every $v\in u+W^{1,1}_0(\Omega')$.
\item[(ii)] 
a \textit{weak quasiminimizer with bound $M>0$} 
of $\mathcal{F}$ if there exists $Q\ge 1$ such that 
\[
\mathcal{F}(u, \Omega'\cap \{u\neq v\}) \le Q \, \mathcal{F}(v,\Omega'\cap \{u\neq v\})
\]
for every open $\Omega' \Subset \Omega$ and every $v \in u+ W^{1,1}_0(\Omega')
$ with $\| v \|_{L^{\infty}(\Omega)}\le M$.
\item[(iii)] 
an \textit{$\omega$-minimizer} of $
\mathcal{F}$ if
there exists a non-decreasing concave function $\omega: [0, \infty) \to 
[0,\infty)$   satisfying $\omega(0)=0$ such that 
\[ 
\mathcal{F}(u, Q_r) \le(1+ \omega(r)) \, \mathcal{F}(v, Q_r)
\]
for every $v \in u+W^{1,1}_0(Q_r)$ with $Q_r \Subset \Omega$.
\item[(iv)] 
a \textit{cubical quasiminimizer} 
of $\mathcal{F}$ if there exists $Q\ge 1$ such that 
\[
\mathcal{F}(u, Q_r) \le Q \, \mathcal{F}(v, Q_r)
\]
for every $v \in u+ W^{1,1}_0(Q_r)$ with $Q_r \Subset \Omega$.
\end{itemize}
\end{defn}

Every minimizer (i.e.\ $1$-quasiminimizer) is both a quasiminimizer and an 
$\omega$-minimizer; and each of these is also a cubical quasiminimizer. 
In addition, it is clear that a quasiminimizer is a weak quasiminimizer with any 
bound $M>0$.
Note that there is no \textit{a priori} relationship between quasiminimizers and 
$\omega$-minimizers: $\omega$-minimizers satisfy a stricter inequality but for a 
restricted range of sets. 

We observe that if $u$ is a quasiminimizer of $\mathcal F$, then it is also 
a quasiminimizer of $\phi + \Lambda$. An analogous result holds for weak quasiminimizers 
and cubical minimizers, but not $\omega$-minimizers. 

To deal with quasiminimizers of $\mathcal F$ we need to 
generalize the results of \cite{HarHT17} which only deal with quasiminimizers of 
$\phi$. It is crucial to track the dependence of constants on $\Lambda$, 
since in Section~\ref{sect:omega} $\Lambda$ depends on the $\omega$-minimizer 
$u$ and may blow up in small balls.  

We record the following iteration lemma, which will be needed in what follows. 

\begin{lem}[Lemma~4.2 in \cite{HarHT17}]\label{Giustin-lemma2}
Let $Z$ be a bounded non-negative function in the  interval $[r,R] \subset \R$ 
and let $X$ satisfy \adec{} on $[0, \infty)$.  
Assume that there exists $\theta\in [0,1)$ such that 
\[
Z(t) \le X(\tfrac1{\sigma-\tau}) + \theta Z(\sigma)
\]
for all $r \le \tau < \sigma \le R$. Then
\[
Z(r) \lesssim X(\tfrac1{R-r}),
\]
where the implicit constant depends only on the \adec{} constants and $\theta$ 
but not on $\|Z\|_\infty$.
\end{lem}

Note that $\|Z\|_\infty$ does not impact the implicit constant in the 
previous result. This will be important for us later on. 

Cubical quasiminimizers need not be bounded in general (cf. \cite[Example~6.5, p.~188]{Giu03}), 
but they do have the following higher integrability property. 

\begin{lemma}[Reverse H\"older inequality]\label{lem:reverse}
Let $\phi\in\Phiw(\Omega)$ satisfy Assumption~\ref{ass:std} and 
suppose $u$ is a cubical quasiminimizer of $\mathcal{F}$.
For any $Q_r \subset \Omega$ with $|Q_r| \le 1$, there exists $s_0>0$ 
such that 
\begin{equation}\label{eq:reverse}
\bigg( \fint_{Q_{\frac{r}{2}}} \phi(x, |\nabla u|)^{1+s_0} 
\,dx\bigg)^{\frac{1}{1+s_0}} 
\lesssim 
\fint_{Q_{r}} \phi(x, |\nabla u|) 
\,dx + \Lambda +1.
\end{equation}
\end{lemma}
\begin{proof}
Consider concentric cubes $Q_{\sigma} \subset Q_{\tau} \subset Q_r$ for 
$0<\sigma<\tau\le r$. 
Let $\eta \in C^{\infty}_0(Q_{\tau})$ be a cut-off function such that $0\le
\eta \le1,\ \eta \equiv 1$ in $Q_{\sigma}$ and $|\nabla\eta| \le\frac{2}{\tau-\sigma}$. 
We use $v := u- \eta \big( u - u_{Q_r} \big)$ as a test function in 
Definition~\ref{defminimizer} (iv) in order to get 
\begin{align}\label{minineq}
\begin{split}
\nu\int_{Q_{\sigma}}\phi(x,|\nabla u|) \,dx 
& \le\int_{Q_{\tau}} F(x, u, \nabla u) \,dx \\
&\le Q \int_{Q_{\tau}} F(x, v, \nabla v) \,dx 
\le 
Q N\int_{Q_{\tau}}\phi(x,|\nabla v|) + \Lambda \,dx.
\end{split}
\end{align} 
We note that $|\nabla v| \le(1-\eta)|\nabla u|+ |\nabla \eta||u-u_{Q_r}| 
\le
2\max\{ (1-\eta)|\nabla u|, |\nabla\eta||u-u_{Q_r}|\}$.
By $|\nabla\eta| \le\frac{2}{\tau-\sigma}$ and \adec{}, we have that 
\[
 \phi(x, |\nabla v|) 
\le 
2^qL \phi(x, (1-\eta)|\nabla u|) + 4^qL\phi\bigg(x, \frac{|u-u_{Q_r}|}
{\tau-\sigma} \bigg).
\]
Denote $c_1:=2^qLQN$. Combining this inequality with \eqref{minineq}, we get that 
\[
\begin{split}
\nu\int_{Q_{\sigma}}\phi(x,|\nabla u|) \,dx 
&\le
c_1\int_{Q_{\tau}}\phi(x, (1-\eta)|\nabla u|)\,dx 
+ c \int_{Q_{\tau}}\phi\bigg(x, \frac{|u-u_{Q_r}|}{\tau-\sigma} \bigg)\,dx 
+ c\Lambda r^n\\
&\le
c_1 \int_{Q_{\tau}\setminus Q_{\sigma}} \phi(x,|\nabla u|) \,dx 
+ c\int_{Q_{\tau}}\phi\bigg(x,\frac{|u-u_{Q_r}|}{\tau-\sigma}\bigg)\,dx + c\Lambda r^n,
\end{split}
\]
where the second inequality follows since 
$\phi\big(x,(1-\eta)|\nabla u|\big) = \phi(x,0) =0 $ in $Q_{\sigma}$. 

Now we use the hole-filling trick and add 
$c_1 \int_{Q_{\sigma}} \phi\big(x,|\nabla u|\big) \,dx$ to both sides of the previous 
inequality and divide by $c_1+\nu$. Then it follows that 
\begin{align*}
 \int_{Q_{\sigma}}\phi(x,|\nabla u|) \,dx 
\le\frac{c_1}{c_1+\nu} \int_{Q_{\tau}} \phi(x,|\nabla u|) \,dx
 + c \int_{Q_{\tau}}\phi\bigg(x,\frac{|u-u_{Q_r}|}{\tau-\sigma}\bigg)\,dx 
+ c\Lambda r^n.
\end{align*} 
By the iteration lemma (Lemma~\ref{Giustin-lemma2}) for the first step 
and the Sobolev--Poincar\'e inequality (Theorem~\ref{thm:SP}) for the second, 
we conclude that 
\begin{equation}\label{eq:cubicalCaccioppoli}
\fint_{Q_{\frac{r}{2}}} \phi(x, |\nabla u|)\, dx 
\lesssim \fint_{Q_r} \phi\bigg(x, \frac{|u-u_{Q_r}|}{r}\bigg)\, dx + \Lambda
\lesssim
\bigg(\fint_{Q_r} 
\phi(x, |\nabla u|)^{\frac 1s}\, dx\bigg)^{s}+\Lambda +1;
\end{equation}
note that the Sobolev--Poincar\'e inequality can be used since 
\[
\int_{Q_r} \phi(x, |\nabla u|)^{\frac 1s} \, dx 
\le 
\int_{Q_r} \phi(x, |\nabla u|) +1\, dx  \le 2.
\]
Hence, by Gehring's lemma (see \cite[Theorem~6.6 and Corollary 6.1, pp.~203--204]{Giu03}), the desired reverse 
H\"older inequality holds. 
\end{proof}

The reverse H\"older inequality has the following ``self-improving'' property. 
\begin{lemma}[Lemma~3.8, \cite{HasO_pp18}]\label{lem:jihoon} 
If $u\in W^{1,\phi}_{loc}(\Omega)$ satisfies \eqref{eq:reverse}, then for every 
$s\in[0,1]$ 
\begin{equation*}
\left(\fint_{Q_r} \phi(x,|\nabla u|)^{1+s_0}\,dx\right)^{\frac{1}{1+s_0}} 
\lesssim
\bigg(\fint_{Q_{2r}} \phi(x,|\nabla u|)^s\,dx\bigg)^{\frac{1}{s}} +\Lambda+1, 
\end{equation*}
where the implicit constant depends on $s$ and the constant in \eqref{eq:reverse}.
If $\phi$ satisfies \adec{}, then this implies that 
\[
\fint_{Q_r} \phi(x,|\nabla u|)\,dx
\le
\left(\fint_{Q_r} \phi(x,|\nabla u|)^{1+s_0}\,dx\right)^{\frac{1}{1+s_0}} \lesssim \phi^+_{Q_{2r}}\left(\fint_{Q_{2r}} |\nabla u|\,dx\right) +\Lambda+1.
\]
\end{lemma}

Let us write 
\[
A(k,r) := Q_r \cap \{u >k\}. 
\]

\begin{lem}[Caccioppoli inequality]\label{lem:Caccioppoli}
Let $\phi\in \Phiw(\Omega)$ satisfy \adec{} and 
let $u$ be a local quasiminimizer of $\mathcal F$. Then 
for all $k\ge 0$ and $0<r<R<\infty$ with $Q_R\subset \Omega$ we have
\begin{equation}\label{eq:caccioppoli}
\int_{A(k, r)} \phi(x,|\nabla (u-k)_+|) \,dx 
\lesssim
\int_{A(k, R)} \phi\Big(x,\frac{(u-k)_+}{R-r}\Big)  \,dx + |A(k,R)|\Lambda.
\end{equation}
\end{lem}

\begin{proof} 
Let $r \le \sigma < \tau \le R$ and $k\ge 0$. 
Let $\eta \in C^\infty_0(Q_\tau)$ be such that 
$0 \le\eta \le 1$, $\eta = 1$ in $Q_\sigma$, and
$|\nabla\eta| \le\frac{2}{\tau-\sigma}$. Denote $v := u-\eta (u-k)_+$. 
Since $u$ is a local quasiminimizer of $\mathcal F$ 
with constant $Q$ and $\spt(u-v) \subset Q_\tau$,
\[
\nu\int_{\{u \neq v\} \cap Q_\tau} \phi(x,|\nabla u|)  \,dx 
\le
QN \int_{\{u \neq v\} \cap Q_\tau} \phi(x,|\nabla v|) + \Lambda \,dx.
\]
Since $A(k,\sigma)\subset \{u\ne v\}\cap Q_\tau\subset A(k,\tau)\subset A(k,R)$, 
this implies that
\[
\int_{A(k,\sigma)} \phi(x,|\nabla u|) \,dx 
\lesssim 
\int_{A(k,\tau)} \phi(x,|\nabla v|) \,dx + |A(k,R)|\Lambda. 
\]
The integrals are handled by the hole-filling trick and the iteration lemma as in 
Lemma~\ref{lem:reverse} (see Lemma~4.3 of \cite{HarHT17} for exact details), while the 
second term on the right-hand side appears directly on the 
right-hand side of the claim. 
\end{proof}


\section{Estimating the essential supremum}\label{sect:esssup}

We now start our proof of Harnack's inequality. As is usual with De Giorgi's method, 
we first derive bounds for the essential supremum of the function. 
In the next section, these will be used to bound also the infimum, 
which combined give the Harnack inequality. 
Recall that $A(k,r)=Q_r\cap \{u>k\}$. 

In this paper we state our results in a modular format so as to make them easier to 
extend later. For instance, in the next result we assume the 
Caccioppoli inequality instead of assuming that $u$ is a quasiminimizer. 
If the Caccioppoli inequality is extended to a larger class, then the next result 
need not be reproved (cf.\ Remark~\ref{rem:weak-quasiminimizers}). 

\begin{lem}\label{HKLMP-lemma_4.3re}
Let $\phi\in\Phiw(\Omega)$ and $u\in W^{1,\phix}_\loc(\Omega)$ 
satisfy Assumption~\ref{ass:std}. 
Suppose that $u$ satisfies the Caccioppoli inequality \eqref{eq:caccioppoli}. 
Let $k\ge 0$ and $0< \sigma < \tau \le R$  with 
$Q_{R}\subset\Omega$ and $(\Lambda+1)|Q_R| \le 1$.
Then
\begin{equation}\label{equ:DeGiorgi-class}
\begin{split}
&\int_{Q_\sigma} \phi\Big(x,\frac{(u-k)_+}{\tau}\Big)\,dx \\
&\qquad\lesssim 
\Big(\frac \tau{\tau-\sigma}\Big)^\frac{q^2}p
\bigg(\frac{|A(k,\tau)|}{|Q_\tau|}\bigg)^\frac 1{2n} 
\bigg(\int_{ Q_\tau} \phi\Big(x,\frac{(u-k)_+}{\tau-\sigma}\Big)\,dx 
+ |A(k,\tau)|(\Lambda +1) \bigg). 
\end{split}
\end{equation}
\end{lem}

\begin{proof}
We first observe that the claim is trivial if $|A(k,\tau)|\ge \frac12 |Q_\tau|$, so 
we may assume that this is not the case.  
Let $\tau':=\frac{\sigma+\tau}2$ and $\eta \in C_0^\infty(Q_{\tau'})$ 
be a cut-off function such that $0 \le\eta \le1$, $\eta = 1$ in $Q_\sigma$, and 
$|\nabla\eta| \le\frac 4{\tau-\sigma}$. Denote $v:=(u-k)_+ \eta$.

By the product rule, $|\nabla v| \le |\nabla (u-k)_+| + (u-k)_+ |\nabla \eta|$. 
Since $|\nabla \eta|\le \frac 4{\tau-\sigma}$, we obtain by \adec{} 
and the Caccioppoli inequality \eqref{eq:caccioppoli} that 
\begin{equation}\label{eq:vsmall}
\begin{split}
\int_{Q_{\tau'}} \phi(x, |\nabla v|)\,dx
&\lesssim 
\int_{Q_{\tau'}}\phi(x,|\nabla (u-k)_+|) + \phi\Big(x,\frac{(u-k)_+}{\tau-\sigma}\Big)\,dx\\
&\lesssim
\int_{ Q_\tau} \phi\Big(x,\frac{(u-k)_+}{\tau-\sigma}\Big)\,dx
+ |A(k,\tau)|\Lambda.
\end{split}
\end{equation}
As an intermediate step, we next show in the case \aone{} how this inequality implies that 
$\rho_\phi(c_{\tau,\sigma}|\nabla v|)\le 1$ for a suitable constant. 

In the case of \aone{}, we denote $w:=(u-k)_+$ and note that $w=0$ in 
$A:=Q_\tau\setminus A(k,\tau)$. 
Since $|A|>\frac 12 |Q_\tau|$, we obtain by the $W^{1,1}$-Poincar\'e inequality, 
Lemma~\ref{lem:aJensen} and $\rho_\phi(|\nabla w|)\le 1$ that 
\[
\begin{split}
|w_A - w_{Q_\tau}|
&\le  \frac{|Q_\tau|}{|A|} \fint_{Q_\tau} |w-w_{Q_\tau}| 
\, dx
\lesssim \tau \fint_{Q_\tau} |\nabla w| \, dx\\
&\lesssim \tau (\phi_{Q_\tau}^-)^{-1}\Big( \fint_{Q_\tau} \phi^-_{Q_\tau}( |\nabla w| ) \, dx \Big)
\le
\tau (\phi_{Q_\tau}^-)^{-1}\Big(\frac1{|Q_\tau|}\Big).
\end{split}
\]
Since $w = 0$ in $A$, we obtain by this, \azero{} and \aone{} that 
\[
w = |w - w_A| 
\le
|w - w_{Q_\tau}| + |w_A - w_{Q_\tau}|
\lesssim
|w - w_{Q_\tau}| + \tau(\phi_{Q_\tau}^+)^{-1}\Big(\frac1{|Q_\tau|}\Big)+\tau.
\]
By the Sobolev--Poincar\'e inequality (Theorem~\ref{thm:SP} with $s=1$), 
\adec{}, \azero{}, $\rho_\phi(|\nabla u|)\le 1$ and $|Q_r|\le 1$, we conclude that 
\begin{equation}\label{eq:areaCase}
\begin{split}
\int_{ Q_\tau} \phi\Big(x,\frac{w}{\tau}\Big)\,dx
&\lesssim
\int_{ Q_\tau} \phi\bigg(x,\frac{|w(x) - w_{Q_\tau}|}{\tau} + (\phi_{Q_\tau}^+)^{-1}\Big(\frac1{|Q_\tau|}\Big)+1\bigg)\,dx\\
&\lesssim
\int_{ Q_\tau} \phi(x,|\nabla u|)+1 +\frac1{|Q_\tau|} \,dx
\le 3. 
\end{split}
\end{equation}
Furthermore, \adec{} implies that 
\[
\begin{split}
\int_{ Q_\tau} \phi\Big(x,\frac{(u-k)_+}{\tau-\sigma}\Big)\,dx
&\lesssim
\big(\tfrac{\tau}{\tau-\sigma}\big)^q
\int_{ Q_\tau} \phi\Big(x,\frac{w}{\tau}\Big)\,dx
\lesssim
\big(\tfrac{\tau}{\tau-\sigma}\big)^q.
\end{split}
\]
By \ainc{p}, \eqref{eq:vsmall} and this imply that 
$\rho_\phi(c_{\tau,\sigma}|\nabla v|)\le 1$, where 
$c_{\tau,\sigma} := c\, (\frac{\tau-\sigma}{\tau})^{q/p}\le 1$.
We set $c_{\tau,\sigma}:=1$ for the case \aonen{}; then in 
both cases we can apply the Sobolev--Poincar\'e inequality (Theorem~\ref{thm:SP})
to the function $c_{\tau,\sigma} v$.

We now start the main line of the proof. 
By Hölder's inequality and \adec{}, we obtain  
\[
\begin{split}
\int_{Q_\sigma} \phi\Big(x,\frac{(u-k)_+}\tau\Big)\,dx
&\le
\int_{Q_{\tau'}} \phi\Big(x, \frac v\tau\Big)\, dx
\le 
|A(k,\tau)|^\frac {s-1}s\bigg(\int_{Q_{\tau'}} \phi\Big(x, \frac{v}{\tau}\Big)^s\, 
dx\bigg)^\frac1s\\
&\le c_{\tau, \sigma}^{-q} 
|A(k,\tau)|^\frac {s-1}s\bigg(\int_{Q_{\tau'}} \phi\Big(x, \frac{c_{\tau, \sigma} v}{\tau}\Big)^s\, 
dx\bigg)^\frac1s\
\end{split}
\]
for $s:=(2n)'$. Note that $s<n'$, $\rho_\phi(c_{\tau,\sigma}|\nabla v|)\le 1$ 
and that $c_{\tau,\sigma} v\in W^{1,\phi}_0(Q_{\tau'})$. 
Thus the Sobolev--Poincar\'e inequality (Theorem~\ref{thm:SP}) for the 
function $c_{\tau,\sigma} v$ yields that  
\[
|Q_\tau|^\frac{s-1}s \bigg(\int_{Q_{\tau'}} \phi\Big(x, \frac {c_{\tau,\sigma} v}\tau\Big)^s\, 
dx\bigg)^\frac1s
\le 
\int_{Q_{\tau'}} \phi(x, |\nabla v|)\, dx + |A(k,\tau)|;
\]
here we also used that $\nabla v=0$ a.e.\ outside $A(k,\tau)$ and 
$c_{\tau,\sigma}\le 1$. 
Combining the two inequalities, noting that $\frac {s-1}s=\frac1{2n}$ 
and using \adec{} for the first step, and \eqref{eq:vsmall} for the second step, 
we find that  
\[
\begin{split}
&\int_{Q_\sigma} \phi\Big(x,\frac{(u-k)_+}\tau\Big)\,dx 
\lesssim
c_{\tau,\sigma}^{-q}\bigg(\frac{|A(k,\tau)|}{|Q_\tau|}\bigg)^\frac1{2n}
\bigg(\int_{Q_{\tau'}} \phi(x, |\nabla v|)\,dx + |A(k,\tau)|\bigg)\\
&\qquad\qquad
\lesssim 
c_{\tau,\sigma}^{-q}\bigg(\frac{|A(k,\tau)|}{|Q_\tau|}\bigg)^\frac 1{2n} 
\bigg(\int_{ Q_\tau} \phi\bigg(x,\frac{(u-k)_+}{\tau-\sigma}\bigg)\,dx + |A(k,
\tau)|(\Lambda +1)\bigg). \qedhere
\end{split}
\]
\end{proof}

Compared to classical estimates, the next proposition contains an extra term $|
u_{Q_r}|$. Note that it involves the function $u$, not just $u_+$, which makes 
it more difficult to manage. 
However, we show that it can be handled in the cases needed to prove Harnack's 
inequality. 
Recall that $q>1$ is the exponent from \adec{q} in 
Assumption~\ref{ass:std}. For brevity, we will use the following 
notation for the rest of the paper 
\[
\lambda_r := (\phi^-_{Q_r})^{-1}(\Lambda +1)\, r.
\]

\begin{prop}\label{prop:esssupinf}
Let $\phi\in\Phiw(\Omega)$ and $u\in W^{1,\phix}_\loc(\Omega)$ 
satisfy Assumption~\ref{ass:std} with $(\Lambda+1)|Q_{r}| \le 1$.
Suppose that $u$ satisfies \eqref{equ:DeGiorgi-class} and $\theta\in [\frac12 ,1)$. 
Then $u_+$ is bounded and 
\begin{equation}\label{eq:supbdd}
\esssup_{Q_{\theta r}} u_+
\lesssim  
(1-\theta)^{-4nq^2} 
\bigg( \bigg[ \fint_{Q_r} (u_+)^q\,dx\bigg]^{\frac1q} + |u_{Q_{r/2}}|\bigg)	+ \lambda_r 
\end{equation}
for any $Q_{2r}\subset \Omega$. The term $|u_{Q_r}|$ can be omitted
if $\big|\{u_+ =0\} \cap Q_r\big| \ge \frac12 |Q_r|$
or if $u$ is non-negative.
\end{prop}
\begin{proof}
For $k>0$ to be chosen and any natural number $j$, we set 
\[
\alpha:=\frac1{2n}, \quad
k_j := r k\Big(1 - \frac{1}{2^{j}}\Big), \quad
\sigma_j := r \Big(\theta + \frac{1-\theta}{2^{j}}\Big),\quad
A_j := A(k_{j+1},\sigma_j),
\]
\[ 
Q_j := Q_{\sigma_j}
\quad\text{and}\quad
Y_j := \fint_{Q_j} \phi \Big(x,\frac{(u-k_j)_+}{r}\Big) \,dx.
\]
Note that $\sigma_j - \sigma_{j+1} = \frac{r(1-\theta)}{2^{j+1}}$. 
Using \eqref{equ:DeGiorgi-class} with $k=k_{j+1}$, $\sigma=\sigma_{j+1}$ and 
$\tau=\sigma_j$ for the middle step, and \adec{} for the others, we find that
\[
\begin{split}
Y_{j+1}
&\lesssim
\fint_{Q_{j+1}} \phi\Big(x,\frac{(u-k_{j+1})_+}{r(\theta+(1-\theta)2^{-j-1})}\Big)\,dx \\
&\lesssim 
2^\frac{q^2j}p (1-\theta)^{-\frac{q^2}p}
\bigg(\frac{|A_j|}{|Q_j|}\bigg)^\alpha
\bigg(\fint_{ Q_j} \phi\Big(x,\frac{(u-k_{j+1})_+}{r(1-\theta) 2^{-j-1}}\Big)\,dx 
+ \frac{|A_j|}{|Q_j|}(\Lambda +1) \bigg)\\
&\lesssim
2^{q^2j} (1-\theta)^{-q^2}
\bigg(\frac{|A_j|}{|Q_j|}\bigg)^\alpha
\bigg(2^{qj}(1-\theta)^{-q} Y_j 
+ \frac{|A_j|}{|Q_j|}(\Lambda +1) \bigg),
\end{split}
\]
where we also used $k_j \le k_{j+1}$ in the last step.
Furthermore, we observe that 
$u-k_j \ge k_{j+1}-k_j = \frac {rk}{2^{j+1}}$ in $A_j$. It follows by \adec{}
that 
\[
\frac{|A_j|}{|Q_j|} 
\le
\fint_{Q_j}\frac1{\phi(x,k)} \phi \Big(x,2^{j+1} \frac{(u-k_j)_+}{r}\Big) \,dx
\lesssim
2^{qj} \phi_{Q_r}^-(k)^{-1} Y_j. 
\]
Now our inequality implies that 
\[
Y_{j+1}
\le c 2^{q^2j} (1-\theta)^{-q^2} (2^{qj} \phi_{Q_r}^-(k)^{-1} Y_j)^\alpha 
\big[2^{qj} (1-\theta)^{-q} Y_j + 2^{qj} \phi_{Q_r}^-(k)^{-1} Y_j(\Lambda+1)\big].
\]
We will choose $k$ such that $\phi_{Q_r}^-(k)^{-1} (\Lambda + 1)\le 
1$. Then the inequality 
implies that 
\[
Y_{j+1}
\le c_1 2^{3q^2j} (1-\theta)^{-2q^2} \phi_{Q_r}^-(k)^{-\alpha} Y_j^{1+\alpha}.
\]

By the well-known iteration lemma \cite[Lemma 7.1, p.~220]{Giu03} if follows that 
$Y_j\to 0 $ as $j \to \infty$, 
provided that $Y_0\le c_1 ^{-1/\alpha} 2^{-3q^2/\alpha^2} 
(1-\theta)^{2q^2/\alpha} \phi_{Q_r}^-(k)
$. 
Thus we need to ensure that 
\[
Y_0 = \fint_{Q_r} \phi \Big(x,\frac{u_+}{r}\Big) \,dx 
\le 
c(\alpha,q) (1-\theta)^{4nq^2}\phi_{Q}^-(k),
\]
which holds under the choice 
\[
\phi_{Q_r}^-(k) =
\frac {\theta_q} {c(\alpha,q)} 
\fint_{Q_r} \phi \Big(x,\frac{u_+}{r}\Big) \,dx + \Lambda 
+ 1,
\quad \theta_q := (1-\theta)^{-4nq^2};
\]
such $k$ exists due to the \adec{} assumption. 
The latter terms are added to ensure that $\phi_{Q_r}^-(k)^{-1} (\Lambda + 1)\le 1$,
as required above. 

Since $k_j\to rk$ and $\sigma_j\to \theta r$ as $j\to\infty$, 
it follows by Fatou's lemma that 
\[
\fint_{Q_{\theta r}} \phi \Big(x,\frac{(u-rk)_+}{r}\Big) \,dx 
\le 
\liminf_{j\to \infty} Y_j = 0. 
\]
This implies that $u\le rk$ a.e.\ in $Q_{\theta r}$. Thus $u$ is locally bounded and 
\begin{equation}\label{eq:sup}
\begin{split}
\esssup_{Q_{\theta r}}  \phi^{-}_{Q_r} \Big(\frac{u_+}r\Big) 
&\le \phi^{-}_{Q_r} (k) 
= 
\frac  {\theta_q} {c(\alpha,q)} \fint_{Q_r} \phi \Big(x,\frac{u_+}{r}\Big) \,dx + 
\Lambda + 1.
\end{split}
\end{equation}

Assume first that $u_{Q_{r/2}}=0$. In the case \aone{}, we use \eqref{eq:sup} 
in the cubes $Q_{r}$ and $Q_{2r}$ (in which case there is no dependence 
on $\theta$ in the constant), the Sobolev--Poincar\'e inequality (Theorem~\ref{thm:SP}) 
with $s=1$ and $u_{Q_{r/2}}=0$, and \adec{} to conclude that
\[
\begin{split}
\esssup_{Q_{\theta r}} \phi^{-}_{Q_r}\Big(\frac{u_+}{r}\Big)
\lesssim 
\esssup_{Q_{r}} \phi^{-}_{Q_r}\Big(\frac{u_+}{r}\Big)
&\lesssim 
\fint_{Q_{2r}} \phi\Big(x, \frac{u_+}{r}\Big)\,dx + 
\Lambda +1 \\
&\lesssim 
\fint_{Q_{2r}} \phi(x, |\nabla u|)\,dx + \Lambda +1
\lesssim 
\frac1{|Q_{r}|},
\end{split}
\]
where in the last step we use $(\Lambda +1)|Q_{r}| \le 1$.
Instead of $u_{Q_{r/2}}=0$ we could assume 
$|\{u_+ =0\} \cap Q_r| \ge \frac12 |Q_r|$
since then \eqref{eq:areaCase} implies that 
\begin{equation*}
\begin{split}
\fint_{Q_{2 r}} \phi\Big(x, \frac{u_+}{r}\Big)\,dx 
&\lesssim 
\frac 1{|Q_r|}.
\end{split}
\end{equation*}
In either case, it follows by  \adec{}, \azero{} and \aone{} that 
\[
\phi\Big(x, \frac{u_+}{r}\Big)
\lesssim 
\phi^{-}_{Q_{r}} \Big(\frac{u_+}{r}\Big) +1
\quad\text{for a.e. } x\in Q_r.
\]
In the case \aonen{}, the same inequality follows from \adec{}, \azero{} and \aonen{},
with constant depending also on $\|u\|_\infty$. Here the assumption 
$u_{Q_{r/2}}=0$ is not needed at all. 

Now we return to \eqref{eq:sup} with $\phi_{Q_r}^-$ in the integral by the estimate 
in the previous paragraph. By \ainc{p} we have
\[
\begin{split} 
\esssup_{Q_{\theta r}} \phi^{-}_{Q_r}\Big(\frac{u_+}{r}\Big) 
&\lesssim 
\frac  {\theta_q} {c(\alpha,q)} \fint_{Q_r} \phi^-_{Q_r} \Big(\frac{u_+}{r}\Big) \,dx + 
\Lambda + 1 
\lesssim 
\fint_{Q_r} \phi^{-}_{Q_r} \Big(\theta_q^\frac1p \frac{u_+}{r}\Big)\,dx + \Lambda +1.
\end{split}
\]
Since $\phi_{Q_r}^-$ is a $\Phi$-prefunction that satisfies \adec{q}, we obtain by 
Lemma~\ref{lem:aJensen} and \eqref{eq:almostId} that
\[
\begin{split} 
\esssup_{Q_{\theta r}} \frac{u_+}{r} 
 \lesssim 
\theta_q^\frac1p \bigg[  \fint_{Q_r} \Big(\frac{u_+}{r}\Big)^{q}\,dx \bigg]^{\frac1q}  
+ \tfrac1r \lambda_r.
\end{split}
\] 
The claim follows for this case when we multiply the previous inequality 
by $r$. 

We have established the claim in the case $u_{Q_{r/2}}=0$. Thus, in the general 
case, 
\[
\esssup_{Q_{\theta r}} u_+-|u_{Q_{r/2}}|
\le 
\esssup_{Q_{\theta r}} (u-u_{Q_{r/2}})_+
\lesssim 
\theta_q^\frac1p \bigg[ \fint_{Q_r} (u-u_{Q_{r/2}})_+^q\,dx\bigg]^{\frac1q} + \lambda_r. 
\]
Furthermore, 
\[
\bigg[ \fint_{Q_r} (u-u_{Q_{r/2}})_+^q\,dx\bigg]^{\frac1q}
\le
\bigg[ \fint_{Q_r} (u_+ + |u_{Q_{r/2}}|)^q\,dx\bigg]^{\frac1q}
\approx 
\bigg[ \fint_{Q_r} u_+^q\,dx\bigg]^{\frac1q} + |u_{Q_{r/2}}|
\]
so we have completed the proof in the general case.
If $u$ is non-negative, then $u_+ = u$ and 
Hölder's inequality allows us to absorb the extra term in the $q$-average as follows:
\[
|u_{Q_{r/2}}| =  \fint_{Q_{r/2}} u \,dx
\lesssim
\bigg[ \fint_{Q_r} u^q\,dx\bigg]^{\frac1q}. \qedhere
\]
\end{proof}

Next we show that the exponent can be decreased arbitrarily 
close to zero when there is no extra term $|u_{Q_{r/2}}|$.

\begin{cor}\label{cor:esssup_nonnegative}
Suppose that $u\in L^\infty(Q_r)$ satisfies \eqref{eq:supbdd} without the term 
$|u_{Q_{\sigma/2}}|$
for $Q_\sigma\subset Q_r$ with $\sigma\in [\frac r2, r]$. 
Then
\[
\esssup_{Q_{r/2}} u_+
\lesssim
 \bigg( \fint_{Q_r} u_+^h\,dx\bigg)^\frac1h +  \lambda_r,
\]
for any $h\in (0,\infty)$. 
The implicit constant depends on $h$ and the constant in \eqref{eq:supbdd}. 
\end{cor}

\begin{proof}
The case $h\ge q$ follows directly by H\"older's inequality, 
so we consider only $h\in (0,q)$. 
Let $\frac r2\le \sigma<\tau\le r$ and denote $Z(\sigma):= \esssup_{Q_{\sigma}} u$. 
By \eqref{eq:supbdd},
\[
\begin{split}
Z(\sigma)
&\lesssim
(1-\tfrac\sigma\tau)^{-4nq^2}
\bigg( \fint_{Q_\tau} u_+^{q}\,dx\bigg)^\frac1{q} + \lambda_\tau 
\le
\Big(\frac r{\tau-\sigma}\Big)^{4nq^2} 
\bigg( \fint_{Q_\tau} u_+^{q}\,dx\bigg)^\frac{1}{q} + \lambda_r.
\end{split}
\]
Since $\tau \in (\frac r2, r)$, we find that 
\[
\bigg( \fint_{Q_\tau} u_+^{q}\,dx\bigg)^\frac1{q} 
\lesssim 
\bigg( \fint_{Q_r} u_+^{h}\, Z(\tau)^{q-h}\,dx\bigg)^\frac1{q} 
\approx
\bigg( \fint_{Q_r} u_+^{h}\,dx\bigg)^\frac1{q}  Z(\tau)^\frac{q-h}q.
\]
Next we use Young's inequality with exponents $\frac{q}h$ and 
$\frac{q}{q-h} =: \frac1\theta $ and obtain
\[
\begin{split}
Z(\sigma)
& \le c \Big(\frac r{\tau-\sigma}\Big)^{4nq^2} 
\bigg( \fint_{Q_r} u_+^{h}\,dx\bigg)^\frac1{q} Z(\tau)^\theta + c \lambda_r \\
&\le 
\underbrace{\frac{c h}{q}\Big(\frac r{\tau-\sigma}\Big)^\frac{4nq^3}h
\bigg(\fint_{Q_r} u_+^{h} \,dx\bigg)^\frac1{h} 
+ c \lambda_r}_{=:X(\frac1{\tau-\sigma})} 
\ +\  \theta Z(\tau) .
\end{split}
\]
Thus $Z(\sigma) \le X(\frac1{\tau-\sigma}) + \theta Z(\tau)$ 
for all $\frac r2 \le\sigma < \tau \le r$. 
Since $Z$ is bounded in $[\frac r2, r]$ and $X$ satisfies \adec{4nq^3/h}, 
Lemma~\ref{Giustin-lemma2} yields $Z(\frac r2) \lesssim X(\frac 2r)$, which is the claim. 
\end{proof}


\section{Estimating the essential infimum}\label{sect:essinf}

Let us denote $D_l := \{u<l\}\cap Q_r$.
Suppose that $u$ is a quasiminimizer of $\mathcal F$ and $l\in \R$. Then 
$l-u$ is a quasiminimizer of 
\[
\int_\Omega G(x,u,\nabla u)\, dx
\quad\text{with}\quad 
G(x,t,z) := F(x,l-t, -z). 
\]
Furthermore, $G$ satisfies \eqref{Hcon} with the same constants 
as $F$. Thus by the Caccioppoli estimate (Lemma~\ref{lem:Caccioppoli}) and 
Lemma~\ref{HKLMP-lemma_4.3re} the function $l - u$ satisfies 
\eqref{equ:DeGiorgi-class}.
Furthermore, the assumption in the next lemma implies that
\[
\big|\{(l-u)_+ =0\} \cap Q_r\big| = |Q_r \setminus D_l| \ge
\big(1 -\tfrac 1{2^qc_1^q}\big)  |Q_r| \ge \tfrac12 |Q_r|,
\] 
so one of the conditions in Proposition~\ref{prop:esssupinf} for omitting the 
term $|u_{Q_r}|$ is satisfied. Thus the implication of the next lemma holds 
in particular for local quasiminimizers. 

\begin{lem}\label{kappa0-lemma}
Let $u\in W^{1,\phix}_\loc(\Omega)$ be non-negative and $l>0$. 
If $l-u$ satisfies \eqref{eq:supbdd} for $\theta=\frac12$ without the term 
$|u_{Q_{r/2}}|$ 
with constant $c_1$, then
\[
|D_l| \le\tfrac 1{2^qc_1^q} |Q_r|
\quad\Rightarrow\quad
\essinf_{Q_{r/2}} u + c_1 \lambda_r \ge \tfrac{l}
{2}.
\]
\end{lem}
\begin{proof}
Inequality \eqref{eq:supbdd} for the function $l-u$ yields that
\[
\begin{split}
\esssup_{Q_{r/2}} (l- u) &\le \esssup_{Q_{r/2}} (l- u)_+
\le  
c_1 \bigg[ \fint_{Q_r} (l- u)_+^q\,dx\bigg]^{\frac1q} 
+ c_1\lambda_r\\
&\le 
 c_1 \bigg[ \frac{1}{|Q_r|}\int_{D_l} l^q\,dx\bigg]^{\frac1q} 
+ c_1\lambda_r
=  c_1 l \bigg[  \frac{|D_l|}{|Q_r|} \bigg]^{\frac1q} + 
c_1 \lambda_r\\
&\le  \tfrac12  l+ c_1 \lambda_r.
\end{split}
\]
Since $\esssup_{Q_{r/2}} (l - u) = l - \essinf_{Q_{r/2}} u$, the claim follows.
\end{proof}

The next lemma shows that the implication of the previous 
lemma holds for any constant $\kappa$. The previous lemma takes 
care of small values of $\kappa$. 

\begin{lem}\label{kappa-lemma}
Let $\phi\in\Phiw(\Omega)$ and $u\in W^{1,\phix}_\loc(\Omega)$ 
satisfy Assumption~\ref{ass:std}. 
Suppose that $u$ is a non-negative local quasiminimizer of $\mathcal{F}$.
Then for every $\kappa \in (0,1)$ there exists $\mu > 0$ such that 
\[
|D_l| \le\kappa\, |Q_r|
\quad\Rightarrow\quad
\essinf_{Q_{r/2}} u + c_1 \lambda_r \ge \mu l
\]
for all $Q_{2r}\subset \Omega$ and all $l>0$. Here the constant $c_1$ is from Lemma~\ref{kappa0-lemma}.
\end{lem}
\begin{proof}
If $l> \|u\|_\infty$, then $|D_l|=|Q_r|$, so there is nothing to 
prove. 
Therefore, we assume that $l\le \|u\|_\infty$.
Abbreviate $Q:=Q_r$ and set, for $0 < h < k < l$, 
\[
v := 
\begin{cases} 
0, & \textrm{ if } u \geq k, \\ 
k-u, & \textrm{ if } h < u < k, \\  
k-h, & \textrm{ if } u \le h.
\end{cases}
\]
Then $v \in W^{1,\phix}_\textrm{loc}(\Omega)$ and $|\nabla v| = |\nabla u| 
\chi_{\{h < u < k\}}$ 
a.e.\ in $\Omega$.

Clearly, $v=0$ in $Q \setminus D_l$, and since $|D_l| \le\kappa|Q|$, 
we have $|Q\setminus D_l| \ge (1-\kappa)|Q|$.  
Under these circumstances, \cite[Theorem~3.16, p.~102]{Giu03} tells us that 
\[
\Big( \int_{Q} v^{n'}\,dx \Big)^\frac1{n'} 
\le
C(n,\kappa) \int_\Delta |\nabla v|\,dx
\]
for $v\in W^{1,1}(Q)$ and $\Delta := D_k\setminus D_h$. 
By Hölder's inequality, 
\[
(k-h) |D_h|^\frac1{n'} = |D_h|^{-\frac1{n}} \int_{D_h} v\,dx 
\le 
\Big( \int_{Q} v^{n'}\,dx \Big)^\frac1{n'}
\lesssim 
|\Delta| \fint_\Delta |\nabla v|\,dx. 
\]
Denote $V(x):=\phi(x,|\nabla v(x)|)$. By H\"older's inequality and 
Lemma~\ref{lem:aJensen}, 
\[
\fint_\Delta |\nabla v|\,dx  
\le 
\Big(\frac{|Q|}{|\Delta|}\Big)^{\frac1{p}}\bigg(\fint_Q |\nabla v|^p\,dx\bigg)^\frac1p  
\lesssim
\Big(\frac{|Q|}{|\Delta|}\Big)^{\frac1{p}}(\phi_Q^-)^{-1}\bigg(\fint_Q V \, dx  \bigg). 
\]

The Caccioppoli estimate (Lemma~\ref{lem:Caccioppoli}) 
for the function $k-u$ implies that 
\begin{equation}\label{eq:Vestimate}
\begin{split}
\fint_Q V\, dx 
&=
\fint_Q \phi(x,|\nabla (k-u)_+|)\, dx
\lesssim
\fint_{Q'} \phi\Big(x,\frac {(k-u)_+}{r}\Big)\,dx + \Lambda
\lesssim
\phi_{Q'}^+\Big(\frac kr\Big)+ \Lambda, 
\end{split}
\end{equation}
where $Q':=Q_{2r}$. 
In the case \aone{}, we use the second expression and the assumption 
$\rho_\phi(|\nabla u|)\le 1$ to conclude that 
$\fint_Q V\, dx \lesssim \frac{1}{|Q|}$. 
It then follows from \aone{}, \azero{} and \adec{} that 
\[
(\phi_Q^-)^{-1}\bigg(\fint_Q V\, dx  \bigg)
\le
(\phi_{Q'}^-)^{-1}\bigg(\fint_Q V\, dx  \bigg)
\lesssim 
(\phi_{Q'}^+)^{-1}\bigg(\fint_Q V\, dx \bigg)+1.
\]
In the case of \aonen{}, we use the last expression of \eqref{eq:Vestimate},
$k\in (0, \|u\|_\infty)$, \azero{} and \adec{} to conclude that 
\[
\fint_Q V\, dx 
\lesssim \phi_{Q'}^+(\tfrac{k}r)+\Lambda
\lesssim \phi_{Q'}^-(\tfrac{k}r)+\Lambda+1
\le
\phi_Q^-(\tfrac{k}r)+\Lambda+1,
\]
where the constant depends on $\|u\|_\infty$.
In either case, we obtain that 
\[
(\phi_Q^-)^{-1}\bigg(\fint_Q V\, dx \bigg)
\lesssim 
\tfrac kr + (\phi_Q^-)^{-1}(\Lambda+1)+1 
\approx
\tfrac 1r (k+\lambda_r),
\]
where we also used \azero{} and \adec{} to absorb the $1$ in $\lambda_r$. 

Combining the previous inequalities, we find that 
\[
\begin{split}
(k-h) |D_h|^\frac{1}{n'} 
&\lesssim
|\Delta| \Big(\frac{|Q|}{|\Delta|}\Big)^{\frac1{p}}\tfrac 1r (k+\lambda_r)
\approx 
|\Delta|^{1-\frac1{p}} r^{\frac n{p}-1} (k+\lambda_r).
\end{split}
\]
We divide the previous inequality by $k$, raise it to the power 
$p'$ 
and substitute $k := l2^{-i}$ and $h := l 2^{-i-1}$, $i\in \N$:
\[
\Big(\frac{l2^{-i}-l 2^{-i-1}}{l2^{-i}}\Big)^{p'} |D_{l 2^{-i-1}}|^\frac{p'}{n'}
\lesssim
\big[|D_{l 2^{-i}}| - |D_{l 2^{-i-1}}|\big] r^{\frac{n-p}{p} p'}
(1+2^i \tfrac1l \lambda_r)^{p'}.
\]

Set $d_i := |D_{l2^{-i}}|$ for $i=0,\ldots, i_0$  and 
note that $\frac{l2^{-i}-l 2^{-i-1}}{l2^{-i}} = \frac12$. 
Since $d_i \geq d_{i_0} = |D_{l 2^{-i_0}}|$ for $i\le i_0$, this implies 
that 
\[
|D_{l 2^{-i_0-1}}|^\frac{p'}{n'}
\lesssim
[d_i - d_{i+1}] r^{\frac{n-p}{p-1}}\big(1+2^{i_0} \tfrac1l \lambda_r\big)^{p'}.
\]
Adding these inequalities for $i$ from $0$ to $i_0-1$, we get
\[
i_0\, |D_{l 2^{-i_0-1}}|^\frac{p'}{n'}
\lesssim
[d_0 - d_{i_0}] r^{\frac{n-p}{p-1}} \big(1+2^{i_0} \tfrac1l \lambda_r\big)^{p'}
\lesssim
r^{n + {\frac{n-p}{p-1}}} \big(1+2^{i_0} \tfrac1l \lambda_r\big)^{p'}.
\]
Now $n + {\frac{n-p}{p-1}} = p\frac{n-1}{p-1} =p'(n-1)$. 
Hence
\[
|D_{l 2^{-i_0-1}}| 
\le
c i_0^{-\frac{n'}{p'}} r^n \big(1+2^{i_0} \tfrac1l \lambda_r\big)^{n'}
=
c_2 i_0^{-\frac{n'}{p'}} |Q| \big(1+2^{i_0} \tfrac1l \lambda_r\big)^{n'}. 
\]
We choose $i_0$ such that $c_2 i_0^{-\frac{n'}{p'}} \le \frac1{2^{q+n'}c_1^q}$ with 
$c_1$ from Lemma~\ref{kappa0-lemma}. 

We consider two cases. If $2^{i_0} \frac1l \lambda_r\le 1$, then the 
previous inequality implies that $|D_{l 2^{-i_0-1}}| \le \frac1{2^qc_1^q}|Q|$, 
in which case it follows from Lemma~\ref{kappa0-lemma} that 
$\essinf_{Q_{r/2}} u + c_1 \lambda_r \ge l 2^{-i_0-1}$, so the 
claim holds with $\mu = 2^{-i_0-1}$. 
If, on the other hand, $2^{i_0} \frac1l \lambda_r\ge 1$, then 
$\essinf_{Q_{r/2}} u + c_1 \lambda_r \ge c_1 2^{-i_0} l$, so the claim holds 
with $\mu=c_1 2^{-i_0}$.
\end{proof}

Now standard arguments yield the weak Harnack inequality, see, e.g., 
\cite[Lemma~6.3]{HarHT17}.

\begin{cor}[Weak Harnack inequality]\label{cor:essinf}
Let $\phi\in\Phiw(\Omega)$ and $u\in W^{1,\phix}_\loc(\Omega)$ 
satisfy Assumption~\ref{ass:std}. 
Suppose that $u$ 
is a non-negative local quasiminimizer of $\mathcal F$. 
Then  there exists $h>0$ such that
\[
 \bigg( \fint_{Q_r} u^h\,dx\bigg)^\frac1h
 \lesssim
 \essinf_{Q_{r/2}} u +  \lambda_r
\]
when $Q_{2r}\subset \Omega$ and $(\Lambda+1) |Q_{2r}| \le 1$.
\end{cor}

By combining Corollaries~\ref{cor:esssup_nonnegative} for 
the non-negative function and \ref{cor:essinf}, we 
obtain Harnack's inequality under Assumption~\ref{ass:std}. 
It remains to be shown that \adec{} can be replaced by \adeci{}.

\begin{proof}[Proof of Theorem~\ref{thm:Harnack}]
Let $\phi$ be from Theorem~\ref{thm:Harnack} and
let $\psi(x,t):=\phi(x,t)+t$. Then, by Lemma~\ref{lem:phi+t}, $\psi$ belongs to 
$\Phiw(\Omega)$ and satisfies Assumption~\ref{ass:std}. In particular, we have 
$\phi \le \psi \lesssim \phi +1$.

Since $u$ is a local quasiminimizer of $\mathcal{F}$, 
it is a local quasiminimizer of $\phi+ \Lambda+1$.
Thus using Corollaries~\ref{cor:esssup_nonnegative} and \ref{cor:essinf} 
with replacing $(\phi, F, \Lambda)$ by $(\psi, \phi+\Lambda+1,\Lambda+1)$, we obtain Harnack's inequality. 
\end{proof}

\begin{remark}\label{rem:weak-quasiminimizers}
All the results in Sections~\ref{sect:quasiminimizers}--\ref{sect:essinf} hold 
also for bounded weak quasiminimizers $u$ with bound $\|u\|_\infty$. 
This follows directly from the given proofs. We use the quasimimimizing property twice, 
first in the proof of the reverse Hölder inequality, Lemma~\ref{lem:reverse}, 
for the test function $v:= u- \eta(u- u_{Q_r}) = (1-\eta) u + \eta u_{Q_r}$, 
and then in the proof of the Caccioppoli inequality, Lemma~\ref{lem:Caccioppoli}, 
for the test function $v:=u- \eta(u- k)_+$, $k \ge 0$. Thus in both cases
$\|v\|_\infty \le \|u\|_\infty$, so we have only used the weak quasiminimizing 
property. In fact, in the proofs that follow, only the latter is needed 
for weak quasiminimizers, the former is applied to the directly for 
cubical quasiminimizers. 
\end{remark}


\section{Morrey estimates}\label{sect:morrey}

It is well known that the Harnack inequality implies the following 
oscillation decay estimate (see \cite[Theorem~8.22]{GilT77} or 
\cite[Theorem~6.6, p.~111]{HeiKM93}). 
We define the oscillation of $u$ by 
$$
\osc (u, r): = \esssup_{Q_r} u - \essinf_{Q_r} u.$$

\begin{thm}[Oscillation decay estimate]\label{Thmosc}
Let $\pm u-k$ satisfy Harnack's inequality for every  $k\in \R$ and 
every $Q_\sigma\subset Q_r$ where it is non-negative. 
Then there exists $\mu\in (0,1)$ such that for all $0< \sigma < r $, 
\[
\osc (u, \sigma) 
\lesssim
\Big(\frac{\sigma}{r}\Big)^{\mu} [\osc(u,r) +\lambda_r ] .
\] 
\end{thm}

In the next theorem we could alternatively use the $p$-average on the left-hand side 
(as in earlier papers like \cite{Ok17}), 
but we use this simpler formulation since it is all we need. 

\begin{thm}[Morrey type estimate]\label{ThmMorreyV}
Let $\phi\in\Phiw(\Omega)$ and $u\in W^{1,\phix}_\loc(\Omega)$ 
satisfy Assumption~\ref{ass:std}. 
Let $u$ be a local quasiminimizer of $\mathcal{F}$.
Then for any $Q_{2r} \subset \Omega$ with $(\Lambda+1) |Q_{2r}| \le 1$,
\[
\int_{Q_{\sigma}} |\nabla u|\,dx 
\lesssim
 \Big(\frac{\sigma}{r}\Big)^{n+\mu -1} \int_{Q_{r}}|\nabla u|+(\phi_{Q_r}^{-})^{-1}(\Lambda+1)  \,dx
\]
for all $0<\sigma<r$, with $\mu$ from Theorem~\ref{Thmosc}. 
\end{thm}

\begin{proof}
It is enough to consider $\sigma \in (0, \frac{r}{4})$. 
By the Caccioppoli inequality (Lemma~\ref{lem:Caccioppoli}) with 
$k=u_{Q_{2\sigma}},r=\sigma,R=2\sigma$, we have that 
\begin{equation*}
\begin{split}
\fint_{Q_{\sigma}}\phi(x,|\nabla (u-u_{Q_{\sigma}})_+|)\,dx 
&\lesssim
\fint_{Q_{2\sigma}} \phi\bigg(x,\frac{|u-u_{Q_{\sigma}}|}{\sigma}\bigg)\,dx +  \Lambda\\
&\le 
\fint_{Q_{2\sigma}} \phi_{Q_{2\sigma}}^+\bigg(\frac{\osc(u, 2\sigma)}{\sigma}\bigg)\,dx 
+ \Lambda 
= 
\phi_{Q_{2\sigma}}^+\bigg(\frac{\osc(u, 2\sigma)}{\sigma}\bigg) + \Lambda.
\end{split}
\end{equation*} 
Since $u$ is a quasiminimizer of $\mathcal{F}$, $-u$ is a quasiminimizer of the functional $\mathcal{F}$
with $F$ replaced by $F(x, -t, -z)$. Hence the Caccioppoli estimate for 
$-u$ similarly implies an estimate for $|\nabla (u-u_{Q_{2\sigma}})_-|$. 
Combining these two estimates we obtain
\begin{equation}\label{eq:morreyProof}
\fint_{Q_{\sigma}}\phi(x,|\nabla u|)\,dx  = \fint_{Q_{\sigma}}\phi(x,|\nabla (u-u_{Q_{\sigma}})|)\,dx 
\lesssim  
\phi_{Q_{2\sigma}}^+\bigg(\frac{\osc(u, 2\sigma)}{\sigma}\bigg) + \Lambda.
\end{equation}

In the case \aone{}, we use Corollary~\ref{cor:esssup_nonnegative} for 
$u-u_{Q_{\tau/2}}$ and $u_{Q_{\tau/2}}-u$ with $h=1$ and
the $W^{1,1}$-Poincar\'e inequality, to derive that 
\begin{equation}\label{eq:oscA1}
\begin{split}
\frac{\osc(u,\tau/2)}{\tau} 
&\le \frac{\sup_{Q_{\tau/2}} (u-u_{Q_{\tau/2}})_+}{\tau}+ 
\frac{\sup_{Q_{\tau/2}} (u_{Q_{\tau/2}}-u)_+}{\tau} \\
&\lesssim
\frac1\tau \fint_{Q_\tau} |u-u_{Q_\tau/2}|\, dx  + \frac1\tau \lambda_\tau
\lesssim
\fint_{Q_\tau} |\nabla u|\, dx  + (\phi_{Q_\tau}^-)^{-1}(\Lambda + 1).\\
\end{split}
\end{equation}
By Lemma~\ref{lem:aJensen}, \adec{}, $\rho_{L^\phi(Q_r)}(|\nabla u|)\le 1$  
and $(1+\Lambda)|Q_r| \le 1$ it follows from this that 
\[
\frac{\osc(u,\tau/2)}{\tau} \lesssim (\phi_{Q_\tau}^-)^{-1}\Big( \frac1{|Q_\tau|} \Big)
\]
for any $0<\tau\leq r.$
We first use this estimate with $\tau = 4\sigma$. 
By \aone{}, \azero{} and \adec{}, we conclude that 
\[
\phi_{Q_{2\sigma}}^+\bigg(\frac{\osc(u, 2\sigma)}{\sigma}\bigg)
\lesssim 
\phi_{Q_{2\sigma}}^-\bigg(\frac{\osc(u, 2\sigma)}{\sigma}\bigg) + 1.
\]
In the case of bounded $u$ and \aonen{}, we obtain the same 
conclusion by \aonen{}, \azero{} and \adec{}, since
$\frac{\osc(u, 2\sigma)}{\sigma} \le\frac{2\| u \|_{L^\infty}}\sigma$.
Thus \eqref{eq:morreyProof} gives 
\[
\fint_{Q_{\sigma}}\phi(x,|\nabla u|)\,dx  
\lesssim  
\phi_{Q_{2\sigma}}^- \bigg(\frac{\osc(u, 2\sigma)}{\sigma}\bigg) + \Lambda+1.
\]

Since $u$ is a local quasiminimizer of $\mathcal F$ with $F(x, t, z)$, 
it follows that $\pm u -k$ is a local quasiminimizer of the functional 
$\mathcal F$  with $F(x, \pm (t +k), \pm z)$. Hence by Theorem~\ref{thm:Harnack} 
we can use Theorem~\ref{Thmosc}. The later theorem and \eqref{eq:oscA1} with 
$\tau = r$ yield:
\begin{align*}
\fint_{Q_{\sigma}}\phi(x,|\nabla u|)\,dx 
 & \lesssim
\phi_{Q_{2\sigma}}^{-}\bigg( \Big( \frac{2\sigma}{r/2}\Big)^{\mu-1} 
\bigg[ \frac{\osc(u, r/2)}{r} + (\phi_{Q_r}^{-})^{-1}(\Lambda+1)\bigg] \bigg) 
+ \Lambda+1\\
 & \approx
\phi_{Q_{2\sigma}}^{-}\bigg( \Big( \frac{\sigma}{r}\Big)^{\mu-1} 
\bigg[ \frac{\osc(u, r/2)}{r} + (\phi_{Q_r}^{-})^{-1}(\Lambda+1)\bigg] 
\bigg) \\
 & \lesssim
\phi_{Q_{2\sigma}}^{-}\bigg( \Big( \frac{\sigma}{r}\Big)^{\mu-1} 
\bigg[ \fint_{Q_r} |\nabla u|\, dx + (\phi_{Q_r}^{-})^{-1}(\Lambda+1)\bigg] 
\bigg),
\end{align*} 
where in the second step we use \eqref{eq:almostId}.
Since $\phi$ satisfies \ainc{1}, Lemma~\ref{lem:aJensen}  and \adec{} imply that 
\begin{align*}
\phi_{Q_{2\sigma}}^- \bigg(\fint_{Q_{\sigma}}|\nabla u|\,dx \bigg)
\lesssim
\fint_{Q_{\sigma}}\phi_{Q_{2\sigma}}^-(|\nabla u|)\,dx 
\lesssim
\fint_{Q_{\sigma}}\phi(x,|\nabla u|)\,dx.
\end{align*} 
We use this on the left-hand side of the earlier estimate 
together with \eqref{eq:almostId} to obtain the claim. 
\end{proof}


\section{Continuity of \texorpdfstring{$\omega$}{omega}-minimizers}\label{sect:omega}

We assume now that the function $F$ satisfies 
\begin{equation*}
\nu\, \phi(x, |z|) \le F(x, t, z) \le N\, \big(\phi(x,|z|) + \Lambda_0 
\big) 
\end{equation*}
for some constant $\Lambda_0\ge 0$. 
Denote $\psi(x,t):= \phi(x,t)+t$. 
By Lemma~\ref{lem:phi+t}, $\psi$ satisfies Assumption~\ref{ass:std}, provided $\phi$ satisfies the assumptions in Theorem~\ref{mainthmV}. Furthermore, 
$W^{1,\phi}=W^{1,\psi}$ since we  consider only bounded domains 
\cite[Corollary~3.3.11]{HarH19}. 

The following is a well known variational principle due to Ekeland; 
see \cite{Eke79} or \cite[Theorem~5.6, p.~160]{Giu03} for its proof. Recall that 
$f: X \to [-\infty, \infty]$  is \textit{lower semicontinuous} 
if $f(v)\le\liminf_{k\to \infty} f(v_k)$
for every sequence $v_k$ convergent to $v \in X$.

\begin{lem}[Ekeland's variational principle]\label{lem:ekeland}
Let $(X,d)$ be a complete metric space and $f: X \to (-\infty, \infty]$ 
be lower semicontinuous with $-\infty < \inf_{X} f < \infty$.
Suppose that
\[ 
f(u) \le\inf_{X}f + \delta
\]
for some $\delta>0$ and $u \in X$. 
Then there exists $w\in X$ with $d(u,w)\le1 $ such that 
\[
f(w)\le f(u) \ \text{ and } \ f(w) \le f(v) + \delta \, d(w, v) \ 
\text{ for all } v \in X.
\]
\end{lem}

We use Ekeland's variational principle in the space
\[
X:=\left\{v \in u + W^{1,1}_0(Q_r) :
\begin{array}{c}
\int_{Q_r} \psi(x,|\nabla v|)\,dx \le\int_{Q_r}\psi(x,|\nabla u|)\,dx\\  
\text{ and }  \ 
\| v  \|_{L^{\infty}(Q_r)} \le \| u \|_{L^{\infty}(Q_r)} \end{array} 
\right\},
\] 
with the metric 
\[
d(v_1, v_2) : = C_r \int_{Q_r} |\nabla v_1 - \nabla v_2| \,dx,
\]
where 
$C_r>0$ is a constant which will be determined later. Moreover we define
$f : X \to \R$ by  $f(v) : = \mathcal{F}(v, Q_r)$.
We first check the assumptions for Ekeland's principle. 

\begin{lemma}\label{equ:metric_space}
Let $\phi\in\Phiw(\Omega)$. 
Then $(X,d)$ is a complete metric space.
If $(t, z)\to F(x, t, z)$ is continuous for every $x$, then $f$ is 
lower semicontinuous.
\end{lemma}

\begin{proof}
It is enough to prove that $(X,d)$ is a closed subspace of $(u + W^{1,1}_0(Q_r), d)$ since $(u + W^{1,1}_0(Q_r), d)$ is a complete metric space.
Let $v_k$ be a sequence in $X$ such that 
$$ \int_{Q_r} |\nabla v_k - \nabla v| \,dx \to 0 \ \textrm{ as } k \rightarrow \infty,$$
for some $v \in u + W^{1,1}_0(Q_r)$.
Then we may assume, passing to a subsequence, if necessary, that 
$v_k \to v$ and $\nabla v_k \to \nabla v$ a.e.\ in $Q_r$. 
By \cite[Lemma~2.1.6]{HarH19}, $\psi(x,\cdot)$ is lower semicontinuous. 
Therefore Fatou's lemma yields that 
\[
\begin{split}
 \int_{Q_r} \psi(x,|\nabla v|)\,dx 
&=  \int_{Q_r} \psi(x,\lim_{k\rightarrow \infty}|\nabla v_k|)\,dx \le \int_{Q_r} \liminf_{k\rightarrow \infty} \psi(x,|\nabla v_k|)\,dx\\
& \le\liminf_{k\rightarrow \infty}  \int_{Q_r} \psi(x,|\nabla v_k|)\,dx \le \int_{Q_r} \psi(x,|\nabla u|)\,dx;
\end{split}
\]
the last step holds since $v_k\in X$. 
We also see that $ \Vert v  \Vert_{L^{\infty}(Q_r)} \leq \liminf_{k\rightarrow \infty}     \Vert v_k \Vert_{L^{\infty}(Q_r)}  \leq   \Vert u \Vert_{L^{\infty}(Q_r)}.$ 
Hence $v \in X$, and so $(X,d)$ is closed.

For the same sequence we have that 
$F(x,v_k(x), \nabla v_k(x)) \to F(x,v(x), \nabla v(x))$ for a.e.\ $x\in Q_r$. 
Then lower semicontinuity follows by Fatou's lemma.
\end{proof}

Notice that a weak quasiminimizer with bound 
$\infty$ is the same thing as a local quasiminimizer. Thus we can cover 
both the bounded and unbounded case with the next lemma, where we show 
that there exists an approximating weak quasiminimizer in every cube $Q_r$. 

\begin{lemma}\label{lemcompV}
Let $\phi\in \Phiw(\Omega)$ satisfy \adeci{} and $(t, z)\to F(x, t, z)$ be continuous. 
Let $Q_{2r} \subset \Omega$ with $|Q_r| \le 1$.
Let $u$ be an $\omega$-minimizer of $\mathcal{F}$.
Then there exists a weak quasiminimizer $w \in u+ W_0^{1,\phi}(Q_r)$ with 
bound $\|u\|_{L^\infty(Q_r)}$ of the functional
\[
\int_{Q_r} \psi(x,|\nabla w|)+ \Lambda\,dx
\qquad \text{with}\quad 
\Lambda := \fint_{Q_r} \phi(x,|\nabla u|)\,dx + \Lambda_0 +1,
\]
satisfying the estimates
\begin{equation}\label{Du<DwV}
\fint_{Q_r} \psi(x,|\nabla w|)\, dx 
\lesssim
\Lambda,
\quad
\|w\|_{L^\infty(Q_r)} \le \|u\|_{L^\infty(Q_r)}
\quad\text{and}
\end{equation} 
\begin{equation}\label{DuDwestV}
 \fint_{Q_r} |\nabla u-\nabla w| \,dx 
\le\omega(r) (\psi_{Q_r}^-)^{-1}(\Lambda).
\end{equation} 
\end{lemma}

\begin{proof}
Let $(X, d)$ and $f$ be as above and 
choose $C_r := [\omega(r) |Q_r| (\psi_{Q_r}^-)^{-1}(\Lambda)]^{-1}$.
For $\epsilon>0$, let $v_{\epsilon} \in X$ be such that 
$f(v_{\epsilon}) \le \inf_{X}f + \epsilon$.
Since $u$ is an $\omega$-minimizer of $\mathcal{F}$,
\begin{align*}
f(u)
&\le(1+\omega(r)) f(v_{\epsilon})\\ 
&\le
\inf_{X} f + \epsilon + \omega(r) \big(\inf_{X} f + \epsilon\big)\\
&\leq
 \inf_{X} f + \epsilon + \omega(r) \bigg(N \int_{Q_r}\phi(x, |\nabla 
u|) +\Lambda_0 \,dx + \epsilon \bigg),
\end{align*}
from which, by letting $\epsilon \to 0^+$ we obtain
\[
f(u) \le\inf_{X} f +\omega(r)N \int_{Q_r} \phi(x, |\nabla u|) + \Lambda_0\,dx 
\le 
\inf_{X} f +\omega(r) N |Q_r| \Lambda.
\]
By Ekeland's principle (Lemma~\ref{lem:ekeland}), there exists $w \in X$ with $f (w) \le f(u)$,
\[
d(u,w) = C_r \int_{Q_{r}} |\nabla u-\nabla w| \,dx \le 1
\] 
and
\[
f(w) 
\le
f(v) + C_r \omega(r) N |Q_r| \Lambda \int_{Q_r} |\nabla w-\nabla v| \,dx 
\]
for all $v \in X$. Note that the former estimate is \eqref{DuDwestV}. 
Furthermore, \eqref{Du<DwV} follows from $w\in X$ and $\psi \lesssim \phi +1$ 
used to estimate:
\[
\fint_{Q_r} \psi(x,|\nabla w|)\, dx  \le \fint_{Q_r} \psi(x,|\nabla u|)\, dx
\lesssim \fint_{Q_r} \phi(x,|\nabla u|) +1\, dx \le \Lambda.  
\] 

It remains to prove that $w$ is a weak quasiminimizer of the $\psi+\Lambda$ energy with 
bound $\|u\|_{L^\infty(Q_r)}$. Let $v \in w+ W^{1,1}_0(Q_r)$ with $\|v\|_{L^\infty(Q_r)}\le \|u\|_{L^\infty(Q_r)}$. 
Assume first that $v \notin X$. 
Since $v$ satisfies the $L^\infty$-bound by assumption, this means that 
$\rho_\psi(\nabla v)>\rho_\psi(\nabla u)$. 
By this and $w \in X$, we have
\[
\int_{Q_r} \psi(x,|\nabla w|)+ \Lambda \,dx 
\le
\int_{Q_r} \psi(x,|\nabla u|)+ \Lambda\,dx 
<
\int_{Q_r} \psi(x,|\nabla v|)+ \Lambda\,dx.
\]
We may cancel the integral over the set $\{w=v\}$, since $\nabla w=\nabla v$ 
a.e.\ in it, so we have the quasiminimizing property in this case. 

It remains to consider the case $v \in X$. 
By the structure conditions on $\mathcal F$, the estimate of $f(w)$ above, 
$\phi\le \psi$, the definition of $C_r$ and the triangle inequality, we 
conclude that 
\[
\begin{split}
\nu \int_{Q_r} \phi(x,|\nabla w|)\, dx 
&\le 
f(w)  
\le
f(v) + C_r \omega(r) N |Q_r| \Lambda \int_{Q_r} |\nabla w - \nabla v|\,dx\\
&\le
N \int_{Q_r} \psi(x,|\nabla v|) + \Lambda_0 \,dx +
\frac{N \Lambda}{(\psi_{Q_r}^-)^{-1}(\Lambda)} 
\int_{Q_r}  |\nabla w| +|\nabla v|\,dx.
\end{split}
\]
By \cite[Lemma~2.2.1]{HarH19}, $\psi^-_{Q_r}$ is equivalent with a 
convex $\xi\in \Phiw$. 
By \cite[Theorem~2.4.10]{HarH19}, we have $\frac{\Lambda}{\xi^{-1}(\Lambda)} 
\approx (\xi^*)^{-1}(\Lambda)$. It follows from Young's inequality, \ainc{1} and 
\eqref{eq:almostId} that 
\[
\frac{\Lambda}{(\psi_{Q_r}^-)^{-1}(\Lambda)} t \approx
(\xi^*)^{-1}(\Lambda) t 
\le
\xi(\epsilon  t) + c_\epsilon \xi^*\big( (\xi^*)^{-1}(\Lambda)\big)
\lesssim 
\epsilon \xi(t) + c_\epsilon \Lambda
\approx
\epsilon \psi_{Q_r}^-(t) + c_\epsilon \Lambda
\]
for any $\epsilon>0$. Using this for $t=|\nabla w|$ and $t=|\nabla v|$ 
as well as the estimate $ \frac1{c_1}\psi(x,|\nabla w|)-1\le \phi(x,|\nabla w|)$, 
we conclude that 
\[
\begin{split}
&\frac 1{c_1} \int_{Q_r} \psi(x,|\nabla w|)\, dx -|Q_r| \\
&\qquad\le \frac{N}{\nu} \int_{Q_r} \psi(x,|\nabla v|) + \Lambda_0 \,dx +
\frac{N\Lambda}{\nu (\psi_{Q_r}^-)^{-1}(\Lambda)} 
\int_{Q_r}  |\nabla w| +|\nabla v|\,dx\\ 
&\qquad \le c_2 \int_{Q_r} \psi(x,|\nabla v|) + \Lambda_0
 + \epsilon \psi_{Q_r}^-(|\nabla w|)
 + \epsilon \psi_{Q_r}^-(|\nabla v|) + c_\epsilon \Lambda\,dx .
\end{split}
\]
We choose $\epsilon$ so small that $c_2\epsilon\le \frac1 {2c_1}$. 
The $\nabla w$-term can be absorbed in the left-hand side and so it follows that
\[
\frac {1}{2c_1} \int_{Q_r} \psi(x,|\nabla w|)\,dx 
\lesssim
(c_2 + \tfrac {1}{2c_1})(c_\epsilon+1) \int_{Q_r} \psi(x,|\nabla v|) + \Lambda \, dx. 
\]
Hence $w$ is a weak quasiminimizer of the $\psi+ \Lambda$ energy.
\end{proof}

Now we are ready to show that $\omega$-minimizers are locally Hölder continuous.

\begin{proof}[Proof of Theorem~\ref{mainthmV}]
Let $Q_{2r} \subset \Omega$ be such that $(\Lambda_0 +1)|Q_{2r}| \le 1$ and
$\rho_{L^\phi(Q_{2r})}(|\nabla u|) \le 1$.
Let $w \in W^{1,\phi}(Q_r)$ be the weak quasiminimizer with bound $\|u\|_{L^\infty(Q_r)}$ from Lemma~\ref{lemcompV}.

Let us first estimate $(\psi_{Q_r}^-)^{-1}(\Lambda)$  and denote $\lambda_0:=(\psi_{Q_r}^-)^{-1} ( \Lambda_0+1)$. By the definition of $\Lambda$, $\phi \le \psi$ and \adec{} we have
\[
(\psi_{Q_r}^-)^{-1}(\Lambda) 
\lesssim
 (\psi_{Q_r}^-)^{-1}\bigg(\fint_{Q_r} \psi(x,|\nabla u|)\, dx\bigg) +\lambda_0.
\]
By $\psi \lesssim \phi +1$ and  $\rho_{L^\phi(Q_{2r})}(|\nabla u|) \le 1$  we have
$\fint_{Q_r} \psi(x,|\nabla u|)\, dx  \lesssim \frac1{|Q_r|}$,
and hence  \azero{}, \aone{}, \adec{} and \eqref{eq:almostId}   yield
\[
(\psi_{Q_r}^-)^{-1}(\Lambda) 
\lesssim
 (\psi_{Q_r}^+)^{-1}\bigg(\fint_{Q_r} \psi(x,|\nabla u|)\, dx \bigg) + \lambda_0.
\]
Since $u$ is a cubical minimizer of $\mathcal{F}$, we may use 
Lemma~\ref{lem:reverse} and thus \eqref{eq:reverse} holds.  
By Lemma~\ref{lem:jihoon}, \adec{} and \eqref{eq:almostId} we conclude that
\[
\begin{split}
(\psi_{Q_r}^-)^{-1}(\Lambda) 
&\lesssim
(\psi_{Q_r}^+)^{-1} \bigg(\psi^+_{Q_r}\Big (\fint_{Q_{r}} |\nabla u|\, dx \Big) +\Lambda_0 +1\bigg) + 
\lambda_0 \approx  \fint_{Q_{r}} |\nabla u|\, dx + \lambda_0.
\end{split}
\]

In the case of \aonen{}, we first use Lemma~\ref{lem:aJensen} with $p=1$, 
then the estimate \eqref{eq:cubicalCaccioppoli}, and finally \azero{} and 
the boundedness of $u$:
\[
\begin{split}
\psi^-_{Q_r} \bigg( \fint_{Q_r} |\nabla u| \, dx \bigg) &\le
\fint_{Q_r} \psi(x,|\nabla u|)\, dx\lesssim 
\fint_{Q_{2r}} \psi\Big(x,\tfrac {|u-u_{Q_{2r}}|}{r}\Big)\, dx + \Lambda_0\\
&\lesssim 
\psi_{Q_{2r}}^-(\tfrac1r \|u\|_{L^\infty(Q_r)}) + \Lambda_0+1
\lesssim (\Lambda_0+1) \psi_{Q_{r}}^-(\tfrac1r).
\end{split}
\]
Thus we have $\fint_{Q_r} |\nabla u| \, dx \lesssim \frac1r$ with implicit constant depending on $\Lambda_0$ 
and hence by  \aonen{}, \adec{} and \azero{} we have 
\[
\psi^+_{Q_r}\Big (\fint_{Q_{r}} |\nabla u|\, dx \Big ) \lesssim \psi^-_{Q_r}\Big (\fint_{Q_{r}} |\nabla u|\, dx\Big ) +1.
\]
Since $u$ is a cubical minimizer of $\mathcal{F}$, we obtain  
by Lemma~\ref{lem:jihoon}, \adec{}, the previous estimate and \eqref{eq:almostId} that
\[
\begin{split}
(\psi_{Q_r}^-)^{-1}(\Lambda) 
&\lesssim
 (\psi_{Q_r}^-)^{-1}\bigg(\fint_{Q_r} \psi(x,|\nabla u|)\, dx\bigg)  + \lambda_0  \\
&\lesssim
 (\psi_{Q_r}^-)^{-1}\bigg( c \psi^+_{Q_r}\Big (\fint_{Q_{r}} |\nabla u|\, dx \Big ) + \Lambda_0 +1\bigg) +\lambda_0  \lesssim  \fint_{Q_{r}} |\nabla u|\, dx + \lambda_0.
\end{split}
\]
Thus we have the same estimate for $(\psi_{Q_r}^-)^{-1}(\Lambda)$ in both cases. 

By \eqref{DuDwestV},  we obtain that 
\[
\begin{split}
\fint_{Q_r} |\nabla u-\nabla w| \,dx 
&\lesssim 
\omega(r) (\psi_{Q_r}^-)^{-1}(\Lambda) 
\lesssim 
\omega(r) \fint_{Q_{r}} |\nabla u| + \lambda_0\, dx.
\end{split}
\]
By Lemma~\ref{lem:aJensen} and \eqref{Du<DwV}, 
\[
\fint_{Q_r}|\nabla w|\,dx 
\lesssim
(\psi_{Q_r}^-)^{-1}\bigg(\fint_{Q_r}\psi(x,|\nabla w|)\,dx \bigg)
\lesssim
(\psi_{Q_r}^-)^{-1}(\Lambda)
\lesssim \fint_{Q_r} |\nabla u| + \lambda_0
\,dx.
\]
On the other hand, from the Morrey estimate (Theorem~\ref{ThmMorreyV}) and 
Remark~\ref{rem:weak-quasiminimizers}, we have, for any $0<\sigma < r$, that
\[
\int_{Q_{\sigma}}|\nabla w|\,dx
\lesssim
\Big( \frac \sigma r\Big)^{n+\mu - 1} \int_{Q_{r}} |\nabla w|+ \lambda_0 \,dx.
\]

Furthermore, since $\mu\in (0,1)$, $\int_{Q_\sigma} \lambda_0\, dx \lesssim 
(\frac \sigma r)^{n+\mu-1}\int_{Q_r} \lambda_0\, dx$. 
Combining these estimates, we find for $0<\sigma<r$, that 
\[
\begin{split}
Z(\sigma):=\int_{Q_{\sigma}}|\nabla u|+\lambda_0\,dx
& \lesssim 
\int_{Q_{\sigma}} |\nabla u-\nabla w| + |\nabla w|+\lambda_0 \,dx \\
& \lesssim
\bigg[ \omega(r) + \Big( \frac \sigma r\Big)^{n+\mu-1}\bigg]
\int_{Q_{r}} |\nabla u| + \lambda_0 \,dx.
\end{split}
\]
Set $\theta:=\frac\sigma r$. Then the previous inequality can be written as 
\[
Z(\theta r) 
\le 
c_1\, \big[\omega(r) + \theta^{n+\mu-1}\big] Z(r).
\]
We first fix $\theta$ such that $c_1\, \theta^{n+\mu-1} 
= \frac12 \theta^{n+\frac\mu2 -1}$. Then we choose $r_0$ so small that 
$c_1 \omega(r) \le \frac12\theta^{n-\frac{\mu}{2}+1}$ when $r\in [0,r_0]$. 
Then the inequality $Z(\theta r) \le \theta^{n+\frac\mu2-1}Z(r)$ 
holds for all $r\le r_0$. 
Thus it follows from \cite[Lemma 7.3, p.~229]{Giu03} that 
\begin{equation*}
\int_{Q_{\sigma}}|\nabla u| + \lambda_0\,dx 
\lesssim
\Big( \frac \sigma r\Big)^{n+\frac{\mu}{2} -1}\int_{Q_{r}}|\nabla u|+ \lambda_0 \,dx
 \end{equation*}
for all $r\le r_0$ and $\sigma \le \tau r$. This and the Poincar\'e inequality imply that 
\begin{equation*}
\sigma^{-n-\frac{\mu}{2}} \int_{Q_{\sigma}}|u-u_{Q_\sigma}| \,dx 
\lesssim
\sigma^{-n-\frac{\mu}{2}+1} \int_{Q_{\sigma}}|\nabla u| \,dx 
\lesssim 
c  
\end{equation*}
for all  cubes $Q_\sigma\subset Q_r$ with $\sigma \le \tau r$. For cubes $Q_\sigma$ with  $\sigma > \tau r$ the claim is trivial.
Thus $u$ belongs  to the Campanato space $\mathcal{L}^{1, n + \frac{\mu}{2}}(Q_r)$.
This implies by the Campanato--Hölder embedding \cite[Theorem~2.9, p.~52]{Giu03} 
that $u \in C^{0, \frac{\mu}{2}}_{\text{loc}}(\overline{Q_r})$.
\end{proof}


\section*{Acknowledgement}

We thank Arttu Karppinen for comments and Jihoon Ok for pointing out 
some flaws in our arguments and helping solve them. We also thank the referee for comments. M. Lee was supported by the National Research Foundation of Korea (NRF) grant funded by the Korea Government (NRF-2019R1F1A1061295).


\bibliographystyle{amsplain}

\begin{thebibliography}{10}



\bibitem{AM2} 
E. Acerbi and G. Mingione: {Regularity results for a class of 
functionals with non-standard growth},
Arch. Ration. Mech. Anal.  156 (2001), no. 2, 121--140.

\bibitem{AM0} 
E. Acerbi and G. Mingione:
{Gradient estimates for the $p(x)$-Laplacean system},
J. Reine Angew. Math. 584 (2005), 117--148.

\bibitem{AhmCGY18}
Y. Ahmida I. Chlebicka, P. Gwiazda and A. Youssfi:
Gossez's approximation theorems in Musielak-Orlicz-Sobolev spaces,
J. Funct. Anal. 275 (2018), no. 9, 2538--2571.


\bibitem{Alm76}
F.J.\ Almgren:
Existence and regularity almost everywhere of solutions to elliptic variational 
problems with constraints. 
{Mem. Amer. Math. Soc.} 4 (1976), no. 165.

\bibitem{Anz83} 
G.\ Anzellotti:
On the $C^{1,\alpha}$-regularity of $\omega$-minima of quadratic functionals, 
{Boll. Un. Mat. Ital. C} (6) 2 (1983), no. 1, 195--212. 


\bibitem{BalD_pp}
A.\ Balci and L.\ Diening: 
New Examples on Lavrentiev Gap Using Fractal, 
Preprint (2019). arXiv:1906.04639


\bibitem{BarCM15} 
P.\ Baroni, M.\ Colombo and G.\ Mingione:
Harnack inequalities for double phase functionals,
Nonlinear Anal. 121 (2015), 206--222.  

\bibitem{BarCM16}
P.\ Baroni, M.\ Colombo and G.\ Mingione:
Non-autonomous functionals, borderline cases and related function classes, 
St Petersburg Math. J. 27 (2016), 347--379.

\bibitem{BarCM18}
P.\ Baroni, M.\ Colombo and G.\ Mingione: 
{Regularity for general functionals with double phase}, 
Calc. Var. Partial Differential Equations  57 (2018),  Paper No. 62, 48 pp.


\bibitem{ByuO17}
S.-S. Byun and J. Oh: 
{Global gradient estimates for non-uniformly elliptic equations},
 Calc. Var. Partial Differential Equations  56  (2017),  no. 2, Paper No. 46, 36 
pp. 

\bibitem{BO1}
S.-S. Byun and J. Ok:
{On $W^{1,q(\cdot)}$-estimates for elliptic equations of $p(x)$-Laplacian 
type},
J. Math. Pures Appl. (9)  106  (2016),  no. 3, 512--545.

\bibitem{ByuRS18}
S.-S. Byun, S. Ryu and P. Shin: 
Calderon-Zygmund estimates for $\omega$-minimizers of double phase variational 
problems,
{Appl. Math. Letters}~\textbf{86} (2018), 256--263.

\bibitem{CapCF18}
C.\ Capone, D.\ Cruz-Uribe and A.\ Fiorenza:
A modular variable Orlicz inequality for the local maximal operator,
Georgian Math. J. 23 (2018), no.~2, 201--206.

\bibitem{Chl18}
I.\ Chlebicka:
A pocket guide to nonlinear differential equations in Musielak--Orlicz spaces,
Nonlinear Anal.\ TMA 175 (2018), 1--27.

\bibitem{ChlGZ18}
I.\ Chlebicka, P.\ Gwiazda  and  A.\  Zatorska-Goldstein: 
Well-posedness of parabolic equations in the non-reflexive and anisotropic Musielak–Orlicz spaces in the class of renormalized solutions,
J. Differential Equations 265 (2018), no. 11, 5716--5766.
    
\bibitem{CheLR06}
Y.\ Chen, S.\ Levine and M.\ Rao:
Variable exponent, linear growth functionals in image restoration,
{SIAM J.\ Appl.\ Math.}~\textbf{66} (2006), no.~4, 1383--1406.

\bibitem{CloGHP_pp18}
A.\ Clop, R.\ Giova, F.\ Hatami and A.\ Passarelli di Napoli:
{Congested traffic dynamics and very degenerate elliptic equations
under supercritical Sobolev regularity},
Preprint (2018). 

\bibitem{ColM15a}
M.\ Colombo and G.\ Mingione: 
Regularity for double phase variational problems,
Arch. Ration. Mech. Anal. 215 (2015), no. 2, 443--496.

\bibitem{ColM15b}
M.\ Colombo and G.\ Mingione: 
Bounded minimisers of double phase variational integrals,
Arch. Ration. Mech. Anal. 218 (2015), no. 1, 219--273.

\bibitem{ColM16}
M.\ Colombo and G.\ Mingione: 
Calder\'on--Zygmund estimates and non-uniformly elliptic operators,
J. Funct. Anal. 270 (2016), 1416--1478.


\bibitem{CruF13}
D.\ Cruz-Uribe and A.\ Fiorenza: \emph{Variable Lebesgue spaces}, Foundations 
and harmonic analysis, Birkh\"auser/Springer, Heidelberg, 2013.

\bibitem{CruH18}
D.\ Cruz-Uribe and P.\ H\"ast\"o: 
{Extrapolation and interpolation in generalized Orlicz spaces}, 
Trans. Amer. Math. Soc. 370 (2018), no. 6, 4323--4349. 

\bibitem{CupGGP18}
G.\ Cupini, F.\ Giannetti, R.\ Giova and A.\ Passarelli di Napoli: 
{Regularity results for vectorial minimizers of a class of
degenerate convex integrals}, 
J. Differential Equations 265 (2018), no. 9, 4375--4416.


\bibitem{DieHHR11} 
L.\ Diening, P.\ Harjulehto, P.\ H\"ast\"o and M.\ R\r u\v zi\v cka:
\emph{Lebesgue and Sobolev spaces with variable exponents}, 
Lecture Notes in Mathematics, 2017. Springer, Heidelberg, 2011.

\bibitem{DuzGG00} 
F.\ Duzaar, A.\ Gastel and J.F.\ Grotowski:
Partial regularity for almost minimizers of quasi-convex integrals,
SIAM J. Math. Anal. 32 (2000), 665--687. 


\bibitem{DolEF96} 
A.\ Dolcini, L.\ Esposito and N.\ Fusco:
$C^{0,\alpha}$ regularity of $\omega$-minima,
Boll. Un. Mat. Ital. A (7) 10 (1996), no. 1, 113--125. 

\bibitem{EspM99} 
L.\ Esposito and G.\ Mingione:
A regularity theorem for $\omega$-minimizers of integral functionals,
Rend. Mat. Appl. (7) 19 (1999), no. 1, 17--44. 

\bibitem{Eke79} 
I.\ Ekeland:
Non convex minimization problems, 
Bull. Amer. Math. Soc., (3) 1 (1979), 443--474.

\bibitem{EleMM16a}
M.\ Eleuteri, P.\ Marcellini and E.\ Mascolo:
Lipschitz continuity for energy integrals with variable exponents, 
Atti Accad. Naz. Lincei Rend. Lincei Mat. Appl. 27 (2016), no. 1, 61–87.

\bibitem{EleMM16b}
M.\ Eleuteri, P.\ Marcellini and E.\ Mascolo:
Lipschitz estimates for systems with ellipticity conditions at infinity, 
Ann. Mat. Pura Appl. (4) 195 (2016), no. 5, 1575–1603.

\bibitem{EleMM18}
M.\ Eleuteri, P.\ Marcellini and E.\ Mascolo:
Regularity for scalar integrals without structure conditions, 
Adv. Calc. Var., to appear. DOI: 10.1515/acv-2017-0037


\bibitem{GiaP13}
F.\ Giannetti and A.\ Passarelli di Napoli:
Regularity results for a new class of functionals with
non-standard growth conditions,
J. Differential Equations 254 (2013) 1280--1305. 

\bibitem{GilT77}
D.\ Gilbarg and N.\ S.\ Trudinger: \emph{Elliptic partial differential equations 
of second order}, Grundlehren der Mathematischen Wissenschaften, Vol. 224. 
Springer-Verlag, Berlin-New York, 1977.

\bibitem{Giu03} E.\ Giusti: 
\emph{Direct Methods in the Calculus of Variations}, World Scientific, 
Singapore, 2003.

\bibitem{GobZ08}
J.\ Goblet and W.\ Zhu:
Regularity of Dirichlet nearly minimizing multiple-valued functions,
J. Geom. Anal. 18 (2008), no. 3, 765--794. 

%

\bibitem{GwiSZ18}
P.\ Gwiazda, I.\ Skrzypczak  and  A.\  Zatorska-Goldstein: 
Existence  of  renormalized  solutions  to  elliptic equation in Musielak-Orlicz space,
J. Differential Equations 264 (2018), no.~1, 341--377.

\bibitem{HarH17}
P.\ Harjulehto and P.\ H\"ast\"o: 
Riesz potential in generalized Orlicz Spaces,
Forum Math. 29 (2017), no. 1, 229--244.

\bibitem {HarH19} 
P.\ Harjulehto and P.\ Hästö: \emph{Orlicz Spaces and Generalized Orlicz Spaces}, 
Lecture Notes in Mathematics, vol. 2236, Springer, Cham, 2019, X+169 pages.
DOI: 10.1007/978-3-030-15100-3. 

\bibitem {HarHK17} 
P.\ Harjulehto, P.\ H\"ast\"o and A.\ Karppinen: 
Local higher integrability of the gradient of a quasiminimizer under 
generalized Orlicz growth conditions,  
Nonlinear Analysis 177 (2018), 543--552.

\bibitem{HarHK16}
P.\ Harjulehto, P.\ H\"ast\"o and R.\ Kl\'en: 
Generalized Orlicz  spaces and related PDE, 
Nonlinear Anal. 143 (2016), 155--173.

\bibitem{HarHLT13}
P.\ Harjulehto, P.\ H\"ast\"o, V.\ Latvala and O.\ Toivanen:
Critical variable exponent functionals in image restoration, 
Appl. Math. Letters 26 (2013), 56--60.

\bibitem {HarHT17} 
P.\ Harjulehto, P.\ H\"ast\"o and O.\ Toivanen: 
Hölder regularity of quasiminimizers under generalized growth conditions, 
Calc. Var. Partial Differential Equations 56 (2017), no. 2, Art. 22, 26 pp.



\bibitem{Has15}
P.\ H\"ast\"o: 
The maximal operator on generalized Orlicz spaces, 
J. Funct. Anal 269 (2015), no. 12, 4038--4048; 
J. Funct. Anal. 271 (2016), no. 1, 240--243.

\bibitem{HasO_pp18}
P.\ H\"ast\"o and J.\ Ok: 
Maximal regularity for non-autonomous differential equations, 
Preprint (2018). arXiv:1902.00261


\bibitem{HeiKM93}
J.\ Heinonen, T.\ Kilpeläinen and O.\ Martio: 
\emph{Nonlinear potential theory of degenerate elliptic equations}, 
Oxford Mathematical Monographs, Oxford Science Publications, 
The Clarendon Press, Oxford University Press, New York, 1993. vi+363 pp.

\bibitem{Kar18}
T.\ Karaman:
Hardy operators on Musielak-Orlicz spaces,
Forum Math. 30 (2018), no.~5, 1245--1254.

\bibitem{KriM05}
J.\ Kristensen and G.\ Mingione:
The singular set of $\omega$-minima,
Arch. Ration. Mech. Anal. 177 (2005), no. 1, 93--114. 


\bibitem{LanM19}
J.\ Lang and O.\ Mendez: 
\textit{Analysis on Function Spaces of Musielak-Orlicz Type},
Chapman \& Hall/CRC Monographs and Research Notes in Mathematics, 2019. 

%


\bibitem{MaeOS17}
F.-Y. Maeda, T. Ohno and T. Shimomura: 
Boundedness of the maximal operator on Musielak-Orlicz-Morrey spaces, 
Tohoku Math. J. 69 (2017), no. 4, 483--495.

\bibitem{MaeMOS13}
F.-Y.\ Maeda, Y.\ Mizuta, T.\ Ohno and T.\ Shimomura: 
Boundedness of maximal operators and Sobolev's inequality on Musielak-Orlicz-
Morrey spaces,
Bull. Sci. Math. 137 (2013), no. 1, 76--96. 

\bibitem{Mar89}
P.\ Marcellini: 
Regularity of minimizers of integrals of the calculus of variations with 
nonstandard growth conditions, 
Arch. Rational Mech. Anal. 105 (1989), no.~3, 267--284.

\bibitem{Mar91}
P.\ Marcellini: 
Regularity and existance of solutions of elliptic equations with $p,q$-growth 
conditions,
{J.\ Differential Equations}~{50} (1991), no.~1, 1--30.

\bibitem{Min06}
G.\ Mingione:
Regularity of minima: An invitation to the dark side of the calculus of 
variations, 
{Appl.\ Math.}~{51} (2006), no.~4, 355--426.
%

\bibitem{Mus83}
J. Musielak: 
\emph{Orlicz spaces and modular spaces},
Lecture Notes in Mathematics, 1034. Springer, Berlin, 1983.



\bibitem{OhnS18}
T.\ Ohno and T.\ Shimomura:  
Maximal and Riesz Potential Operators on Musielak--Orlicz Spaces Over Metric Measure Spaces,
Integral Eqations Operator Theory 90 (2018), no.~ 6, article 62.

\bibitem{Ok16}
J.\ Ok: 
Gradient estimates for elliptic equations with $L^{p(\cdot)}\log L$ growth, 
Calc. Var. Partial Differential Equations 55 (2016), no. 2, 1--30.

\bibitem{Ok16b}
J.\ Ok: 
Regularity results for a class of obstacle problems with nonstandard growth, 
J. Math. Anal. Appl. 444 (2016), no. 2, 957--979.

\bibitem{Ok16c}
J.\ Ok: 
{Harnack inequality for a class of functionals with non-standard growth via 
De Giorgi's method},
Adv. Nonlinear Anal. 7 (2018), no. 2, 167-182.

\bibitem {Ok17} 
J.\ Ok:  
Regularity of $\omega$-minimizers for a class of functionals with 
non-standard growth, 
Calc. Var. Partial Differential Equations 56 (2017), no. 2, Art. 48, 31 pp. 

\bibitem{Orl31} 
W.\ Orlicz: 
{\"Uber konjugierte Exponentenfolgen}, 
Studia Math. 3 (1931), 200--211. 


\bibitem{RadR15}
V.\ Radulescu and D.\ Repovs:
\textit{Partial Differential Equations with Variable Exponents: Variational Methods and Qualitative Analysis},
Chapman \& Hall/CRC Monographs and Research Notes in Mathematics, 2015.


\bibitem{RafS17}
H.\ Rafeiro and S.\ Samko: 
Maximal operator with rough kernel in variable Musielak-Morrey-Orlicz type 
spaces, variable Herz spaces and grand variable Lebesgue spaces,
Integral Equations Operator Theory 89 (2017), no. 1, 111--124. 

\bibitem{RaoR91}
M.\ M.\ Rao and Z.\ D.\ Ren:
\emph{Theory of Orlicz spaces}, 
Monographs and Textbooks in Pure and Applied Mathematics, 146. Marcel Dekker, 
Inc., New York, 1991.

\bibitem{Ruz00}
M.\ R\r u\v zi\v cka: 
\textit{Electrorheological fluids: modeling and mathematical theory}, 
Lecture Notes in Mathematics, 1748. Springer-Verlag, Berlin, 2000.



\bibitem{YanLK17}
D.\ Yang, Y.\ Liang and L.\ Ky: 
\textit{Real-variable theory of Musielak-Orlicz Hardy spaces}. 
Lecture Notes in Mathematics, 2182. Springer, Cham, 2017. xiii+466 pp.

\bibitem{YanYZ14}
D.\ Yang, W.\ Yuan and C.\ Zhuo:
Musielak-Orlicz Besov-type and Triebel-Lizorkin-type spaces,
{Rev.\ Mat.\ Complut.}~{27} (2014), no.~1, 93--157.

\bibitem{ZhaR18}
Q.\ Zhang and V.\ R\u{a}dulescu: 
Double phase anisotropic variational problems and combined effects of reaction and absorption terms,
J. Math. Pures Appl. (9) 118 (2018), 159--203.

\bibitem{Zhi86}
V.V.\ Zhikov: 
Averaging of functionals of the calculus of variations and elasticity theory, 
Izv. Akad. Nauk SSSR Ser. Mat. 50 (1986), 675--710.


\bibitem{Zhi95}
V.V.\ Zhikov: 
On Lavrentiev's Phenomenon, 
Russian J. Math. Phys. 3 (1995), 249--269.


\end{thebibliography}

\end{document}